\documentclass[reqno,10pt]{amsart}
\usepackage{amsfonts,amssymb,latexsym,epsfig,amsmath}
\usepackage{amsthm}
\usepackage{dsfont}
\usepackage{color}
\newtheorem{theorem}{Theorem}[section]
\newtheorem{proposition}[theorem]{Proposition}
\newtheorem{corollary}[theorem]{Corollary}
\newtheorem{definition}[theorem]{Definition} %[section]
\newtheorem{lemma}[theorem]{Lemma} %{Lemma}
\theoremstyle{remark}

\newtheorem{remark}[theorem]{Remark}
\numberwithin{equation}{section}
\newcommand{\N}{\mathbb{N}}

\newcommand{\R}{{\mathbb R}}

\newcommand{\tr}{\hbox{\rm tr}}

\newcommand{\bx}{{\bar x}}
\newcommand{\by}{{\bar y}}

\newcommand{\eps}{\varepsilon}
\newcommand{\ra}{\rightarrow}
\newcommand{\alp}{\alpha}

\newcommand{\ind}[1]{\mathds{1}_{#1}}
\newcommand{\dmu}{\,\mathrm{d}\mu}

\newcommand{\hyp}[1]{{\bf $($#1$)$}} 
\newcommand{\M}[2]{\mathop{\R^{#1\times #2}}}
\renewcommand{\d}{\mathrm{d}}

\newcommand{\Meps}{${\bf M}_\eps$}
\newcommand{\dz}{\,\mathrm{d}z}

\newcommand{\ds}{\,\mathrm{d}s}

\newcommand{\dy}{\,\mathrm{d}y}
\newcommand{\ol}{\overline}
\newcommand{\ul}{\underline}
\newcommand{\cT}{\mathcal{T}}
\newcommand{\brak}[1]{{\rm [\,#1\,]}}
\renewcommand{\sup}{\mathop{\mathstrut\textrm{sup}}}
\renewcommand{\inf}{\mathop{\mathstrut\textrm{inf}}}

\begin{document}
\title[]
{On nonlocal quasilinear equations\\ and their local limits}

\author[E. Chasseigne \&  E. R. Jakobsen]
{Emmanuel Chasseigne \& Espen R. Jakobsen}

\address{Emmanuel Chasseigne\newline
    Laboratoire de Math\'ematiques et Physique Th\'eorique (UMR CNRS 7350)\newline
    F\'ed\'eration Denis Poisson (FR CNRS 2964)\newline
    Universit\'e F. Rabelais - Tours\newline
    Parc de Grandmont\newline
    37200 Tours, France}
\email{emmanuel.chasseigne@univ-tours.fr}
\urladdr{http://www.lmpt.univ-tours.fr/$\sim$manu}

\address{Espen R. Jakobsen \newline
    Department of Mathematical Sciences \newline
    Norwegian University of Science and Technology \newline
    7491 Trondheim, Norway }
\email{erj@math.ntnu.no}
\urladdr{http://www.math.ntnu.no/$\sim$erj}

\date{\today}
\thanks{E.C. is partially supported by Spanish Project MTM2011-25287, and
  E.R.J. is partially supported by the NFR Toppforsk project Waves and Nonlinear Phenomena (project 250070).}

\keywords{Nonlocal elliptic equation, quasilinear equations,
  quasilinear nonlocal operators, 
   viscosity solutions, L\'{e}vy processes, infinity-Laplace,
   p-Laplace, nonlocal parabolic equation, existence, uniqueness,
   local limits.} 

\subjclass[2010]{35R09, 45K05, 35J60, 35J62, 35J70, 35K59, 47G20,
  35D40, 35A01, 35B51, 35B40}

%	35J62  	Quasilinear elliptic equations
%       35J70  	Degenerate elliptic equations
% 35K55  	Nonlinear parabolic equations
%35K59  	Quasilinear parabolic equations
%35K65  	Degenerate parabolic equations
%35D40  	Viscosity solutions
%35R09  	Integro-partial differential equations
%45K05  	Integro-partial differential equations
%47G20  	Integro-differential operators 
%35A01  	Existence problems: global existence, local existence,
%non-existence 
%35A02  	Uniqueness problems: global uniqueness, local
%uniqueness, non-uniqueness 
%35B51  	Comparison principles
%35B40  	Asymptotic behavior of solutions

\begin{abstract}
We introduce a new class of quasilinear nonlocal operators and study
equations involving these operators. The operators are degenerate
elliptic and may have arbitrary growth in the gradient. Included are new
nonlocal versions of $p$-Laplace, $\infty$-Laplace, mean curvature
of graph, and even strongly degenerate operators, in addition to some nonlocal
quasilinear operators appearing in the existing literature. Our main 
results are comparison, uniqueness, and existence results
for viscosity solutions of linear and fully nonlinear equations involving
these operators. Because of the structure of our operators, especially
the existence proof is highly non-trivial and non-standard.
We also identify the conditions under which the
nonlocal operators converge to local quasilinear operators, and show
that the solutions of the corresponding nonlocal equations converge to
the solutions of the  local limit equations. Finally, we give a (formal)
stochastic representation formula for the solutions and provide many examples.  

% DROP???: Loosely speaking our results imply that for ``any'' quasilinear 2nd
% order local operator, ``any'' well-posed local equation, and ``any''
% nonlocal 
% L\'evy type operator , there is a corresponding L\'evy type quasilinear
% operator and a well-posed nonlocal equation. Moreover, the
% solution of any such local equation  can be approximated by the
% solutions of a multitude of different nonlocal equations. 
\end{abstract}

\maketitle

%------------------------------------------------------------------
\section{Introduction} \label{sect:introduction}
%------------------------------------------------------------------

In this paper we introduce a new class of gradient dependent L\'evy type
diffusion operators and study the well-posedness, stability, and some
asymptotic behavior of equations involving such operators. The
operators we will consider are the following, 
$$L[u,Du]=(L_1+L_2)[u,Du]$$
where
\begin{align}
\label{I-def}
L_1[u,Du](x)&=\int_{\R^P} u\big(x+j_1(Du,z)\big)-u(x)
-j_1(Du,z)\cdot Du(x)\, \dmu_1(z)\,,\\ \label{I-def2}
L_2[u,Du](x)&=\int_{\R^P} u\big(x+j_2(Du,z)\big)-u(x)
\dmu_2(z)\,,
\end{align}
and $\mu_1,\mu_2$ are non-negative L\'evy measures and $j_1,j_2$ are
measurable functions (see Section \ref{sect:main}). %In these operators 
Here the strength and direction of the diffusion depend on the
gradient, and hence as we explain below, these operators are natural
generalizations of the local (non-divergence form) quasilinear operators
%of the form 
\begin{align*}%\label{L0}
L_0(Du,D^2u)=\frac12\tr\big(\sigma(Du)\sigma(Du)^TD^2u\big)+b(Du)Du.
\end{align*}
The operators are allowed to degenerate ($j_1=0$ or $j_2=0$ in some set)
and have arbitrary growth in the gradient, so $\infty$-Laplace, $p$-Laplace,
and strongly degenerate operators are included. Included are also
``explicit'' operators of the form (cf. Section \ref{sect:flex}),
\begin{align}
\label{flex}a(Du)\Big[-(-\Delta)^{\frac\alpha2}u\Big]\qquad\text{for all}\qquad
\alpha\in(0,2)\ \ \text{and}\ \ a\in C(\R^N;\R^+).
\end{align}

We want to study equations involving the operator $L$, and to simplify
and focus on the new issues, the main part of this paper is devoted
to the following special problem: 
\begin{equation}\label{eq:0}
    F\big(u,Du,L[u,Du]\big)=f(x)\qquad\text{in}\qquad \R^N,
\end{equation}
where we assume $F$ to be (degenerate) elliptic and strictly increasing in
$u$ (i.e. $D_uF>0$). But for this equation, we make an effort to push for
very general results. First we obtain comparison,
uniqueness, stability, and existence results for bounded solutions of
\eqref{eq:0}. These results are highly non-trivial due to the
implicit nature of our operators and our weak integrability
assumptions. Especially existence is very challenging as we discuss below. 
%Especially the proof of the existence result is difficult
%and non-standard.
% Because of
%the structure of our operators, classical arguments based on
%compactness or Perrons method seem not to work.
We then identify the limit problems where 
nonlocal operators converge to local ones,
\begin{align*}
L_\eps[\phi,D\phi]\to
L_0(D\phi,D^2\phi) \qquad\text{as}\qquad \eps\to 0,
\end{align*}
for any smooth and bounded function $\phi$,
and prove that the solutions $u_\eps$ of the corresponding nonlocal
equations 
\begin{equation}\label{eq:eps}
    F\Big(u_\eps,Du_\eps,L_\eps[u_\eps,Du_\eps]\Big)=f(x)\qquad\text{in}\qquad \R^N,
\end{equation}
converge locally uniformly to the solution of the local equation
\begin{equation}\label{eq:00}
    F\Big(u,Du,L_0(Du,D^2u)\Big)=f(x)\qquad\text{in}\qquad \R^N.
\end{equation}
We refer to Section \ref{sect:main} for the
precise assumptions and results, and to Section \ref{sect:extensions}
for extensions to more general problems like parabolic
problems and problems with several nonlocal operators. Here we just
remark that $(i)$ the weak  
solution concept we use is bounded viscosity solutions, $(ii)$
generators $L$ of every pure jump L\'evy processes are included as 
linear special cases, and $(iii)$ a typical special case of \eqref{eq:0}
satisfying our assumptions is the quasilinear equation 
\begin{equation}\label{eq:main.nonlocal}
  -L[u,Du](x)+u(x)=f(x)\qquad\text{in}\qquad \R^N\,,
\end{equation}
with bounded uniformly continuous $f$.

Let us illustrate our results on $\infty$-Laplace type
operators.  In the local case (e.g. \cite{Li:Note}) this operator has
``diffusion'' (Brownian motion, generator $(-\Delta$)) only in the
gradient direction:    
\begin{align}
\label{inf-lap-loc}
\Delta_\infty
u(x) %=Du(x)^TD^2u(x)Du(x)
=\tr[Du(x)Du(x)^TD^2u(x)]=(Du(x)\cdot D)^2 u(x).
\end{align}
Natural nonlocal generalizations 
%which is covered by our results,
are operators with e.g. $\alp$-stable 
diffusion ($\alp\in(0,2)$) %fractional diffusion 
along the gradient direction. The generator of the symmetric
$\alp$-stable process is the fractional Laplacian \cite{A:Book},
$$-(-\Delta)^{\alp/2}u(x)=\int_{\R^N} u\big(x+z)-u(x)
-\big(z\cdot Du(x)\big)\,\ind{|z|<1}\frac{c_\alp \dz}{|z|^{N+\alp}}\,,$$
and hence the corresponding nonlocal version of the $\infty$-Laplace
operator would take the form
\begin{align}
\label{inf-lap}
\mathcal{L}_{\Delta_\infty}^{\alpha/2}[u](x)=\int_{\R^1} u\big(x+Du(x)z\big)-u(x)
-Du(x)\cdot Du(x)z\,\ind{|z|<1}\frac{c_\alp \dz}{|z|^{1+\alp}}\,. 
\end{align}
This operator is in the form $L$ with $j_1=Du\cdot z=j_2$,
$\mu_1=\ind{|z|<1}\mu$, and $\mu_2=\ind{|z|\geq 1}\mu$, where %and $\mu$ is such that
$d\mu=\frac{c_\alp
  \dz}{|z|^{1+\alp}}$.  
By our results, $L=\mathcal{L}_{\Delta_\infty}^{\alpha/2}$ gives
rise to well-posed equations \eqref{eq:0}, and since
$$\mathcal{L}^{\alp/2}_{\Delta_\infty}[\phi](x)\to 
\Delta_\infty \phi(x)\qquad\text{as}\qquad\alp\to2^-$$
for smooth bounded $\phi$, it also follows that (possibly non-smooth
viscosity) solutions of \eqref{eq:eps} with
$L_\eps=\mathcal{L}^{1-\eps}_{\Delta_\infty}$ will converge as
$\eps\to0$ to the solution of \eqref{eq:00} with
$L_0=\Delta_\infty$. 

 {\em A similar construction can be carried out for
``any'' local (non-divergence form) quasilinear operator and 
``any'' L\'evy diffusion, thereby producing a corresponding quasilinear
L\'evy diffusion}. Under our assumptions this new operator is
well-posed, and can 
approximate the original local operator. This will be explained
in Remark \ref{rem:lim}. In Section \ref{sect:examples} we present a
(formal) stochastic interpretation of our equations and give many more
examples. Included are several nonlocal versions of the
$\infty$-Laplace, the $p$-Laplace, and the  
mean curvature of graph operators; versions that are modulations
of singular integral operators and others based on bounded
nonlocal operators. 
It is interesting to note that the limit operator $L_0$
will include also a drift term ($b\neq0$) whenever the measures
$\mu_{2,\eps}$ in the $L_2$-term has a non-zero mean value near
$z=0$, see assumption \hyp{M$_\eps$} in section
\ref{ssec:loclim}. The reason is that in $L_2$ this mean is not 
compensated by a first-order gradient term as in $L_1$.
 
The literature on nonlocal equations is very large, and we will
restrict the following discussion to nonlocal quasilinear problems and the
questions that we address in this paper: Well-posedness, stability and
asymptotic limits. We will not discuss important
issues such as regularity of solutions or numerical
algorithms. In the literature, typically the nonlocal quasilinear
operators either have ``coefficients'' depending on $u$ or on $Du$
(but see also \cite{CLM2012}). In the former case you find e.g. all
the equations of porous medium type, see e.g. \cite{CV,DPQRV,BIK,EJDT} and 
references therein. The second case is the case that we consider in
this paper. Here the literature seems to be rather
recent. In the calculus of variations, such equations can be obtained
as Euler-Lagrange equations by minimizing fractional Sobolev norms
($W^{p,\frac\alp2}$-norms) \cite{LL2014,DCKP2013,L2014} or truncated
versions of such norms \cite{AMRTBook}. In the first three papers,
(variational) fractional $p$ and $\infty$-Laplace operators are
introduced. In \cite{IN2010}, a different
``variational'' type of nonlocal operators is studied by
non-variational viscosity solution techniques. In one space dimension,
non-variational equations of the type
\begin{align}
u_t+|u_x|^{m}(-\Delta)^{\frac\alp2}u=0\qquad\text{in}\qquad \R^1\times(0,T)
\end{align}
have been studied with viscosity solution techniques in
e.g. \cite{IMR,SDTV} for different values of $m>0$ and
$\alp\in(0,2)$. Such equations are motivated either by
dislocation dynamics or porous medium flow, and along with
their natural extensions to arbitrary space dimensions, they belong to
the class of equations we study here (cf. sections \ref{sect:flex} and \ref{sect:extensions}). Non-variational nonlocal $\infty$-Laplace type
operators are introduced in \cite{BCF2012_2,BCF2012}, and shown in
\cite{BCF2012_2} to be connected to a sequence of Tug of War games. 
But none of these operators
have an implicit form as our operators do. 
Our operators are not variational, and among existing (multi-dimensional)
work they resemble most closely the operators of
\cite{BCF2012_2,BCF2012}, especially \cite{BCF2012_2}. 
However, whereas the operators in \cite{BCF2012_2,BCF2012}
have bounded dependence on the gradient but are discontinuous
where it is zero, our operators are continuous but may have arbitrary
growth in the gradient. The operators in \cite{BCF2012_2,BCF2012}
correspond to normalized $\infty$-Laplacians, which in the local case
take the form (see e.g. \cite{PSSW,Li:Note})
$$\frac1{|Du(x)|^2}\Delta_\infty u(x)=\Big(\frac{Du(x)}{|Du(x)|}\cdot D\Big)^2u(x),$$
while our version \eqref{inf-lap} corresponds to an unnormalized one
(i.e. to $\Delta_\infty u$).
 
In this paper we work with viscosity solutions. This weak solution
concept is not distributional and does not involve integration. It is
very well adapted to the implicit and degenerate form of our equations.
The viscosity solution concept was introduced by Crandall and Lions in the early
1980s to get uniqueness of solutions of first order Hamilton-Jacobi
equations. Later it has been extended to wide rage of problems,
including many nonlocal ones. The standard reference for local
problems is \cite{cil}. For nonlocal problems, we only refer to
\cite{BI,JK} for the basic well-posedness theory for problems posed in
the whole space. But we mention that there is a large literature on
regularity and properties of solutions, asymptotic problems, boundary
conditions, 
approximations and numerics, relation to stochastic processes,
applications etc..
The problems we consider here represent a natural class of nonlocal
quasilinear equations where the viscosity solution techniques still
apply and give comparison and uniqueness.

In fact we have optimized
the assumptions to allow for very general dependence on the gradients
in $L$ and $F$ at the cost of no dependence on the variable $x$! We have also made
an effort to optimize the assumption on $j_i$ and $\mu_i$. In both
cases our assumptions are much more general than in \cite{BI,JK}. In
the doubling of variables argument of the comparison proof, these
differences to \cite{BI,JK} are e.g. reflected in a different choice
of test function and two of the limits being taken in the reverse
order. Reversing the limits is contrary to most viscosity solutions
proofs, but it is essential in our proof. A side effect is that
$\frac{|\bar x-\bar y|^2}{\eps}\not\to 0$ and hence that we cannot
consider equations with non-trivial $x$-dependence. %  Here you have to 
% choose: either general $x$-dependence or the general quasilinear
% dependence that we study here.
Existence, on the other hand,
does not follow from clever modifications of commonly used
arguments. Because of the implicit form of the equations, with the
gradient dependence in $j_1$ and $j_2$, compactness arguments do not
work and it seems not possible to adapt Perron's method
  either. Instead we propose a new
  argument based on a so-called 
  \textit{Sirtaki method} inspired by \cite{BCGJ}. It involves several
regularization and approximation arguments, a Schauder fixed point
argument, and several limit problems. In each limit problem, we obtain
a limit solving the relevant limit equation by the half
relaxed limit method combined with strong comparison results. The
argument is non-standard and highly non-trivial.

In section
\ref{sect:extensions}, we give the extension to the parabolic  case
(Cauchy problems) and to problems with many nonlocal operators
including e.g. Bellman-Isaacs type  equations. A natural open question is to study
less degenerate equations without the assumption that $D_uF>0$, like
 uniformly elliptic or even $p$ and $\infty$-Laplace equations. 
Another one is to consider such
equations on domains with boundary conditions.  Finally, we mention that in 
 an upcoming  paper %\cite{ACJ}, 
we will study the local limits of
 nonlocal equations  under assumptions that are optimized w.r.t. the
 $x$-dependence. In  this case we also give explicit convergence rates.

\medskip

% \textbf{Related works to be commented --- } 
% Discussion of related literature...
% \begin{itemize}
%     \item[$(a)$] Julio \cite{AMRTBook} : regular measure, approx of $p$-laplace
%          and $\infty$-Laplace; but they treat $z\in\R^N$.
%     \item[$(b)$] Ishii-Nakamura \cite{IN2010} : fractional $p$-Laplace;
%     \item[$(c)$] Bjorland, Caffarelli,  Figalli \cite{BCF2012}; and \cite{BCF2012_2}
%     \item[(iv)] Erik et al \cite{CLM2012} : variational $p$-laplace.
%     \item[(v)] Erik ++ \cite{LL2014,L2014,DCKP2013}, fractional
%       eigenvalue, fractional p-Laplace.
%     \item[(vi)] nonlocal viscosity solution theory....
% \end{itemize}

% \medskip

\subsection*{Outline} We present the main results in
Section~\ref{sect:main} and give several examples and a stochastic
interpretation in Section~\ref{sect:examples}. Then, precise
definitions of viscosity solutions 
appear in Section~\ref{sect:visco} and the proofs of the comparison,
existence and concentration results are given in
Section~\ref{sect:proofs}. In Section~\ref{sect:extensions} we extend
our results to parabolic problems and problems with many nonlocal
operators, and in the appendix at the end of the paper, we give the
proofs of some technical results we need. 

\medskip

\subsection*{Notation} 
The notation $UC(\R^N)$ denotes the set of uniformly continuous
functions defined on $\R^N$ and $BUC(\R^N)$ is the space of bounded, uniformly
continuous functions; $usc$ [resp. $lsc$ ] stands for upper semicontinuous [resp.
lower semicontinuous]; the spaces $C^1 / C^2$ are the spaces of functions
having continuous first-order / second-order derivatives; 
$C^{0,\alpha},C^{1,\alpha}$ stand for the usual H\"older spaces; $C_b$ denotes
the space of continuous, bounded functions;
$\limsup^*$ and $\liminf_*$ are the half-relaxed limits (more precise
definitions in the text where they are used); we denote by $\ind{A}$ the
indicator function of the set $A$; a modulus of continuity is a subadditive
function $\omega:\R_+\to\R_+$ such that $\lim_{s\to 0^+}\omega(s)=0$; the
notation $a\wedge b$ stands for the min of $a$ and $b$, $a\vee b$ is for the
max and $s^+=\max(s,0)$. Note that in this paper $x\in\R^N$ for $N\geq1$ while
$z\in\R^P$ for $P\geq1$; finally, $\M{P}{Q}$ denotes the space of matrices
with $P$ rows and $Q$ columns.

%-----------------------------------------------------------------
\section{The main results} \label{sect:main}
%-----------------------------------------------------------------

The results of this section
essentially implies that for ``any'' quasilinear 2nd order local operator
$L_0$, ``any'' well-posed local equation \eqref{eq:00}, and ``any'' nonlocal
L\'evy type operator, there is a corresponding L\'evy type quasilinear
operator $L$ 
and a {\em well-posed} nonlocal equation \eqref{eq:0}. Moreover, the
solution of any such local equation  can be {\em approximated locally
uniformly} by the solutions of a multitude of different nonlocal equations.

%----------------------------------------------
\subsection{Comparison, uniqueness, and existence}

Let us first list the assumptions under which we construct a general existence and
uniqueness theory for \eqref{eq:0}:

\noindent\hyp{M} $\mu_1$ and $\mu_2$ are non-negative Radon measures on $\R^P\setminus\{0\}$ 
satisfying
$$\int_{|z|>0}|z|^2\dmu_1(z)+\int_{|z|>0}\dmu_2(z)<\infty\,.$$

\noindent\hyp{J1}  $j_1(p,z)$ and $j_2(p,z)$ are Borel measurable
functions from $\R^N\times\R^P$ into $\R^N$, continuous in $p$ for
a.a. $z\in\R^P$, and for any $r>0$ there is a $C_{j,r}>0$ such that
for all $|p|<r$,   
$$\int_{|z|>0} |j_1(p,z)|^2\dmu_1(z)\leq C_{j,r}.$$

\noindent\hyp{J2}
For any $r>0$, there is a modulus of continuity $\omega_{j,r}$ such
that for all $|p|,|q|<r$, 
    $$\int_{|z|>0}|j_1(p,z)-j_1(q,z)|^2\dmu_1(z)\leq 
    \omega_{j,r}(p-q)\,.$$

    \noindent\hyp{J3} There exists $\delta_0>0$ such that for any
    $r>0$ and $\eps>0$ there exists $\eta>0$ such that
    $$\sup_{|p|<r}\int_A|j_1(p,z)|^2\,\d\mu_1(z)<\eps\,$$
 for every Borel set $A\subset \{0<|z|<\delta_0\}$ such that
 $\int_A|z|^2\mu_1(dz)<\eta$.  

% , 
% the family of functions $g_p:z\mapsto |j_1(p,z)|^2$ is $\mu_1$-equi-integrable
% in $B_{\delta_0}(0)$ with respect to $|p|\leq r$: For 
%     any

\noindent\hyp{F1} $F:\R\times\R^N\times\R\to\R$ 
is continuous, and for any $u\in\R$, $p\in\R^N$, 
$\ell\leq
\ell'$,  $F(u,p,\ell)\geq F(u,p,\ell')\,.$

\noindent\hyp{F2} For any $M>0$, there exist $\gamma_M>0$ such that
for all $p\in\R^N$, $l\in\R$, and $-M\leq v\leq u\leq M$,% such that $u\geq v$, 
$$F(u,p,l)-F(v,p,l)\geq \gamma_M(u-v)\,.$$

\noindent\hyp{F3} For any $M,r>0$, there exists a modulus of 
continuity $\omega_{M,r}$ such that for any $|u|\leq M$
and $|p|,|q|,|\ell|,|\ell'|\leq r$,
$$\Big|F\Big(u,p,\ell\Big)-
F\Big(u,q,\ell'\Big)\Big|\leq \omega_{M,r}\big(|p-q|+|\ell'-\ell|\big)\,.$$ 

\noindent\hyp{F4} $f\in UC(\R^N)$.

\noindent\hyp{F5} $f\in BUC(\R^N)$ and all quantities in \hyp{F2},
\hyp{F3} are independent of $M$.

We give now the precise results and refer to Section~\ref{sect:proofs} for
the proofs.
	
\begin{theorem} {\rm (Comparison results)}\

    \noindent $(a)$ \brak{Quasilinear case} \label{thm:comp}
Assume \hyp{M}, \hyp{J1}--\hyp{J2}, and \hyp{F4}.
If $u:\R^N\to\R$ is a bounded usc subsolution of \eqref{eq:main.nonlocal} and
$v:\R^N\to\R$ is a bounded lsc supersolution of \eqref{eq:main.nonlocal}, then
  $u\leq v$ in $\R^{N}$.\medskip 

  \noindent $(b)$ \brak{Fully nonlinear case}
    Assume \hyp{M}, \hyp{J1}--\hyp{J2},
    \hyp{F1}--\hyp{F4} hold.
If  $u:\R^N\to\R$ is a bounded usc viscosity subsolution of \eqref{eq:0} and
$v:\R^N\to\R$ is a bounded lsc viscosity supersolution of \eqref{eq:0}, then
$u\leq  v$ in $\R^N$.
\end{theorem}

We have the following immediate consequences of this comparison result.

\begin{corollary}
\label{cor:main}
Under the assumptions of Theorem \ref{thm:comp}:

\noindent $(a)$ \brak{Uniqueness} There is a most one solution $u\in C_b(\R^N)$ of
\eqref{eq:0} (respectively of \eqref{eq:main.nonlocal}). 

\noindent $(b)$ \brak{Uniform continuity} Any solution $u\in C_b(\R^N)$ of
\eqref{eq:0} (respectively of \eqref{eq:main.nonlocal}) belongs to
$BUC(\R^N)$ and
$$\gamma_M\,\omega_u(h)\leq
\omega_f(h)\qquad (\text{respectively}\ \omega_u(h)\leq
\omega_f(h)),$$
where $M=\|u\|_\infty$ and
$\omega_\phi(r)=\displaystyle\sup_{x\in\R^N,|y|<r}|\phi(x+y)-\phi(x)|$
denotes the modulus of continuity of $\phi(x)$.

\noindent $(c)$ \brak{$L^\infty$-bound} If also \hyp{F5} holds with 
$\gamma_M=\gamma$ (independent of $M$), then any solution $u\in C_b(\R^N)$ of
\eqref{eq:0} (respectively of \eqref{eq:main.nonlocal}) satisfies 
$$\gamma\|u\|_\infty\leq\|f\|_\infty\qquad(\text{respectively}\ \|u\|_\infty\leq\|f\|_\infty). $$
\end{corollary}

\begin{proof}
$(a)$ is immediate from Theorem \ref{thm:comp}, while $(c)$ follows since
$\pm\frac1\gamma\|f\|_\infty$ are super and subsolutions of
\eqref{eq:0}. To prove $(b)$, note that
$v_\pm(x)=u(x+h)\pm\frac1{\gamma_M}\omega_f(|h|)$
is a super and subsolution of \eqref{eq:0}. By Theorem
\ref{thm:comp}, $v_-(x)\leq u(x)\leq v_+(x)$, and hence
$|u(x)-u(x+h)|\leq \frac1{\gamma_M}\omega(|h|)$. 
\end{proof}

\begin{theorem}[Existence] \label{thm:existence}
    Under the assumptions of Theorem
    \ref{thm:comp}, \hyp{J3}, and \hyp{F5}, there exists a
    unique bounded viscosity solution $u\in BUC(\R^N)$ of \eqref{eq:0}
    (respectively of \eqref{eq:main.nonlocal}).
\end{theorem}

Let us now briefly comment on the assumptions.

\begin{remark} 
\label{rem:assump}
$(i)$ $\mu_1$ and $\mu_2$ are L\'{e}vy
measures \cite{A:Book} by \hyp{M}. Conversely, any L\'evy measure
$\mu$ can be written as $\mu_1+\mu_2$ for $\mu_1$ and $\mu_2$
satisfying \hyp{M}: 
$$\mu=\mu\mathds{1}_{|z|<1}+\mu\mathds{1}_{|z|\geq1}=:\mu_1+\mu_2\,.$$
\noindent $(ii)$ Assumptions on $j$ are optimized w.r.t. the
dependence in $p$ at the cost of no dependence on $x$! A typical
example is  
$$j_i(p,z)=j(p)z,$$
 where $z\in\R^P$ and
 $j:\R^{N}\to\M{N}{P}$ only needs to be continuous. 
 \hyp{J1} and \hyp{J3} follow 
from the stronger assumption $|j_i(p,z)|\leq c|z|$ for
$z$ near $0$. \hyp{J3} implies
that $\big\{\frac{|j_1(p,z)|^2}{|z|^2}\big\}_{|p|<r}$ is
$|z|^2\mu_1(dz)$ equi-integrable on
$\{0<|z|<\delta\}$ for any $\delta\leq\delta_0$, \textit{cf.} Appendix
\ref{app:equiint}. We need it to construct solutions under our general
assumptions but not for comparison.

\noindent $(iii)$ By \hyp{M}, \hyp{J1}, and a Taylor
expansion, $L[\phi,D\phi](x)$ is well-defined for any 
$\phi\in C^2(\R^N)\cap C_b(\R^N)$. 

\noindent $(iv)$ \hyp{F2} implies degenerate ellipticity and strict
monotonicity in $u$, while \hyp{F3} allows for very general $p$-dependence at
the cost no $x$-dependence. Compare \hyp{F3} to e.g.
assumption (3.14) in \cite{cil}. 

\noindent $(v)$ The assumptions on integrability and $p$-dependence of
$j$ and the $(p,l)$-dependence of $F$ of this paper are much more
general than e.g. in \cite{JK,BI}. 
\end{remark}

%----------------------------------------------------
\subsection{Local limits} 
\label{ssec:loclim}
We also study the convergence of solutions of
the nonlocal equation \eqref{eq:eps} to the local equation
\eqref{eq:00}, including separate results for the quasilinear case where
\eqref{eq:eps} and \eqref{eq:00} take the simpler forms 
\begin{align}
\label{eq:lineps}
-L_\eps[u_\eps,Du_\eps](x)+ u_\eps(x) &= f(x)\quad\text{in}\quad \R^N\,,\\
\label{eq:local}
            -L_0(Du,D^2u)+
             u(x) &= f(x)\quad\text{in}\quad \R^N\,,
\end{align}
where the local operator $L_0$ is precisely defined in Definition~\ref{defsigma} below.
Concerning \eqref{eq:lineps}, we use the decomposition
$L_\eps=L_{1,\eps}+L_{2,\eps}$ with 
$$L_{i,\eps}[\phi,D\phi](x)=\int_{\R^P}
\phi(x+j_i(D\phi(x),z)-\phi(x)- 
       \delta_{i,1}\, j_i(D\phi(x),z)\cdot
       D\phi(x)\ d\mu_{i,\eps},\qquad\text{$i=1,2$},$$
where $\delta_{i1}=1$ if $i=1$ and $0$ otherwise. In order to prove the
convergence result as $\eps\to0$ we need the following additional 
assumptions: 
\smallskip

\noindent\hyp{\Meps} $\mu_\eps=(\mu_{1,\eps},\mu_{2,\eps})$
satisfies \hyp{M} for every $\eps>0$,
and there exists $A_1,A_2\in\M{P}{N}$ and $a\in\R^P$ such that for every
$Y\in\M{P}{P}$, $q\in \R^P$, and $\delta>0$,  as $\eps\to0$
\begin{align*}
&\int_{|z|<\delta}z^T Yz\,  \dmu_{1,\eps}\to \tr[A_1^T YA_1],\\%\quad
&\int_{|z|<\delta} (z^TYz+ q\cdot z)\, \dmu_{2,\eps}\to \tr[ A_2^T
YA_2]+a \cdot q,\\ %\quad
&\int_{|z|>\delta}(\dmu_{1,\eps}+\dmu_{2,\eps})\to0.
\end{align*}

\noindent\hyp{J4} For $i=1,2$, the function $(p,z)\mapsto j_i(p,z)$ is continuous,
$z$-differentiable at $z=0$ locally uniformly in $p$, $j_i(p,0)=0$,
and the function $\sigma_i:\R^N\to\M{N}{P}$, defined by
$$\sigma_i(p):=D_zj_i(p,0),\quad\text{is continuous.}$$

\begin{definition}\label{defsigma} For any vector $p\in\R^N$ and 
    matrix $X\in\M{N}{N}$, we define:
    $$L_0(p,X):=\frac{1}{2}\tr\big[\tilde\sigma_1(p)\tilde\sigma_1(p)^TX\big]+
    \frac{1}{2}\tr\big[\tilde\sigma_2(p)\tilde\sigma_2(p)^TX\big]+b(p)\cdot
    p,$$
where $\tilde\sigma_i(p):=\sigma_i(p)A_i$, $i=1,2$, $b(p):=a^T\sigma_2(p)$, for
$A_1$, $A_2$, $a$, $\sigma_1$, $\sigma_2$  given by \hyp{M$_\eps$}
and \hyp{J4}. 
\end{definition}

The limit result is the following:
\begin{theorem} {\rm (Local limits)} 
\label{thm:localization}
Let $L_0$ be given by Definition \ref{defsigma}. 

\noindent $(a)$ \brak{Quasilinear case} Assume \hyp{\Meps},
\hyp{J1}--\hyp{J4} and \hyp{F5}. Then any sequence of solutions  $u_\eps$ of
\eqref{eq:lineps} converges locally uniformly as $\eps\to0$ to the solution $u$ of
\eqref{eq:local}.\\

\noindent $(b)$ \brak{Fully nonlinear case} Assume \hyp{\Meps},
\hyp{J1}--\hyp{J4}, and
\hyp{F1}--\hyp{F5}. Then any sequence of solutions  $u_\eps$ of
\eqref{eq:eps} converges locally uniformly as $\eps\to0$ to the solution $u$ of
\eqref{eq:00}.\\
\end{theorem}

\begin{remark}\label{rem:lim}
$(i)$ \hyp{\Meps} is a concentration assumption implying
e.g. $z^TYz\,\mu_{1,\eps}(\dz)\rightharpoonup \tr[A_1^TYA_1]
\delta_0$ in measure. This is a convergence
result for measures in $\R^P$ and not in $\R^N$. 
Note that $a$ plays a role only for the $L_2$-part of
$L$. Illustrative examples are the following singular and truncated 
$(2-\eps)$-stable like L\'{e}vy measures: 
\begin{align*}
\mu_{1,\eps}(\d
z)&=\eps\frac{g(z)}{|z|^{N+2-\eps}}\mathds{1}_{|z|<1}\d
z\qquad\text{where}\qquad \lim_{z\to 0}g(z)=g(0)\neq 0,\\
\mu_{2,\eps}(\d
z)&=\eps\frac{g(z)}{|z|^{N+2-\eps}}\mathds{1}_{\eps<|z|<1}\d z
\qquad\text{where}\qquad g \text{ is $C^1$ at $z=0$ and } g(0)\neq0.
\end{align*}
Both satisfy \hyp{\Meps}: $\mu_{1,\eps}$ with $A_1=g(0)I$, $A_2=0$, $a=0$ and
$\mu_{2,\eps}$ with $A_1=0$, $A_2=g(0)I$, $a=Dg(0)$.% (Taylor expand $g$). 

\noindent $(ii)$ By \hyp{J4} and Definition \ref{defsigma},
$$L_0(p,X)=\tr[A_1Y_1A_1]+\tr[A_2Y_2A_2]+a\cdot q$$
for $Y_i=\sigma_i(p)^TX\sigma_i(p)\in\M{P}{P}$, $i=1,2$, and
$q=\sigma_2(p)p\in\R^P$. 

\noindent $(ii)$ If \hyp{F5} and \hyp{J4} (and \hyp{F1}--\hyp{F3}) hold,
there exists a unique viscosity solution of \eqref{eq:local} (and of
\eqref{eq:00}) satisfying the strong comparison principle, cf. Theorem
5.1 in \cite{cil} and Lemma \ref{comp_loc} below. 

 \noindent $(iii)$ We may specify (``any'') $\sigma(p)$ first, and then for
 every L\'evy measure $(\mu_{1,\eps},\mu_{2,\eps})$ satisfying the
 concentration assumption \hyp{M$_\eps$}, we get a nonlocal
 approximation $L_\eps$ of the local operator $L_0$. Moreover, the
 corresponding equations, \eqref{eq:eps} and \eqref{eq:0}, are
 well-posed with solutions that converge to one another under very
 general assumptions.
\end{remark}

%-----------------------------------------------------------------------------------
\section{Stochastic interpretation and examples}\label{sect:examples}
%-----------------------------------------------------------------------------------

\subsection{Stochastic interpretation}
Formally equation \eqref{eq:0} is always the Dynamic Programming
Equation of an implicitly defined stochastic control problem or
game. E.g. the solution $u$ of \eqref{eq:main.nonlocal} satisfies
formally 
\begin{equation}
    u(x)=\mathbb{E}^x\left(\int_0^\infty e^{-t}f(X_t)dt\right)\label{u_def}
\end{equation}    
where $X_t$ is a pure jump L\'evy-Ito process satisfying
\begin{equation}
  X_t=x+\int_0^t\int_{|z|>0}j_1(Du(X_{s^-}),z) \tilde N_1(\dz,\ds) +
\int_0^t\int_{|z|>0}j_2(Du(X_{s^-}),z) N_2(\dz,\ds)\,,\label{X_def}
\end{equation}
where $\tilde N_1$ is a compensated Poisson random measure, $N_2$ is a
finite intensity Poisson random measure, 
and $\mathbb{E}^x$ is the expectation w.r.t. the law of $X$ (which starts at
$x$). By a L\'{e}vy-Ito
process we mean a {\em L\'{e}vy type stochastic integral} defined in
Chapter 4.3.3 in \cite{A:Book}, and we refer to
e.g. \cite{A:Book,CT:Book} for definitions of the other probabilistic
terms mentioned above. Formally, the generator $A$ of $X_t$ is given
by the formula
$$Au(x)=L[u,Du](x)$$
for $u$ in the domain of $A$ (equation
(6.36) in \cite{A:Book}). Moreover, $X_t$ generates a semigroup $T_t$ 
defined by $T_t\phi(y)=\mathbb{E}^y\phi(X_t^y)$ with the convention that 
$X_0^y=y$ almost surely, and $u$ is then the $1$-resolvent $R_1$ (chapter 3 in
\cite{A:Book}) of this the semi-group applied to $f$, \textit{i.e.}
$R_1f(x)=1u(x)$. By 
the resolvent identity,  
$$u-Au=(I-A)R_1f(x)=f(x)\qquad\text{in}\qquad\R^N,$$
i.e. $u$ satisfies equation \eqref{eq:main.nonlocal} at least
formally. To make this discussion rigorous, we need the assumptions of section
\ref{sect:main} and some additional ones including smoothness
of $u$. Following chapter 6.7 in \cite{A:Book}, it suffices to assume in
addition that Assumptions 6.6.1 and 6.7.1 of \cite{A:Book} hold. We
do not state them here, we only remark that they are satisfied if e.g. 
$$j_1(p,z)=j(p)z\ \text{with}\ j\in W^{1,\infty}_{\mathrm{loc}},\qquad 0\leq
\mu(\dz)\leq \frac {C\, 
    \dz}{|z|^{N+\alp}}\ \text{with}\ \alp\in(0,2), \qquad u\in
C^2_0\ \text{and}\ f\in C_0\,.$$
Note that then $Du$ is bounded and Lipschitz. In this case it follows
from Theorem 6.7.4 of \cite{A:Book} that $T_t$ is a Feller semi-group
with generator $A$ as above and that $u$ is in the domain of $A$. By
Theorem 3.2.9 of \cite{A:Book} the resolvent $R_1$ exists and satisfies
the resolvent identity above for any $f\in C_0(\R^N)$.

We have the following result:
\begin{proposition}
If $u$ and $X_t$ satisfy the assumptions mentioned above and
\eqref{u_def} and \eqref{X_def} hold (in the 
strong sense), then 
$u$ is a classical solution of \eqref{eq:main.nonlocal}.   
\end{proposition}

\subsection{Isotropic operators involving the fractional Laplacian}
\label{sect:flex}
We will explain why products of the fractional Laplacian
and a positive scalar function of the gradient (cf. \eqref{flex}),
are operators of the type we consider here in this paper.

By the scaling properties of the Levy measure $\frac{c_{N,\alp}}{|z|^{N+\alp}}dz$ and a change of variables,
$$-a^\alp(-\Delta)^\alp u(x)=\int_{\R^N} u\big(x+az\big)-u(x)
-az \cdot Du(x)\ \frac{c_{N,\alp} dz}{|z|^{N+\alp}}\qquad\text{for
  all}\qquad a\geq0,\ \alp\in(0,2),$$
and hence for every $x$,
$$-a(Du(x)) (-\Delta)^\alp u(x)=\int_{\R^N} u\big(x+a^{\frac1\alp}(Du(x))z\big)-u(x)
-a^{\frac1\alp}(Du(x))z \cdot Du(x)\ \frac{c_{N,\alp} dz}{|z|^{N+\alp}}.$$
It is immediate that assumptions \hyp{M}, \hyp{J1}--\hyp{J3}
are all satisfied for this
operator when $\alp\in(0,2)$ and $a\in C(\R^N;\R^+)$.

In one space dimension and with $a(p)=|p|^{m-1}$, $m>0$, this operator
appear in models of dislocations in crystals ($m=1$) \cite{Head,IMR,BKM},
and in certain nonlocal porous medium models ($m>1$) \cite{SDTV} as
the (integrated) equation for the cumulative distribution function. 

\subsection{Examples} 
% The moral is the following: For ``any'' local
% quasilinear operator $L_0(Du,D^2u)$ as in \eqref{L0}, and ``any'' Levy
% type nonlocal operator, we can define a corresponding quasilinear
% Levy type operator \eqref{I-def} e.g. by taking
% $j(p,z)=\sigma(p)z$. Moreover, 
% equations based on such nonlocal operators will be well-posed just as
% their local counterparts.
We introduce now some classes of quasilinear nonlocal operators with special focus
on operators of $p$-Laplacian, $\infty$-Laplacian, and
mean curvature of graph type. Recall the definitions of the local and
fractional $\infty$-Laplacian in \eqref{inf-lap-loc} and
\eqref{inf-lap}. To define other nonlocal operators we need the
following Lemma. 
\begin{lemma}Let $p\geq 1$, $r_p=-1+\sqrt{p-1}$ and $I$ be $N\times N$ identity matrix.

    \noindent $(a)$ $\displaystyle I+(p-2) \frac{\xi\otimes
  \xi}{|\xi|^2}=a_p(\xi)a^T_p(\xi)\qquad\text{where}\qquad a_p(\xi):=I+r_p\frac{\xi\otimes
  \xi}{|\xi|^2}.$

\noindent $(b)$ $\displaystyle I- \frac{\xi\otimes
  \xi}{1+|\xi|^2}=\tilde a(\xi)\tilde
a^T(\xi)\qquad\text{where}\qquad\tilde a(\xi):=I-\frac{\xi\otimes 
  \xi}{|\xi|^2}\Big(1-\frac{1}{\sqrt{1+|\xi|^2}}\Big).$

\noindent $(c)$ The functions $\tilde a(\xi)$ and $|\xi|^{\frac{p-2}{2}}a_p(\xi)$, $p\geq2$, are
continuous in $\R^N$. 
\end{lemma}
The proof is straightforward, using that $r_p^2+r_p=p-2$. 
In view of the lemma, 
%it is easy to
%see that:
\begin{align*}
%\Delta_\infty u(x)&=Du(x)^TD^2u(x)Du(x)=\tr[Du(x)Du(x)^TD^2u(x)],\\[0.1cm]
\Delta_pu(x)&=\mathrm{div}\Big(|Du(x)|^{p-2}Du(x)\Big)=|Du(x)|^{p-2}\Big(\Delta
u(x)+(p-2)\frac{\Delta_\infty u(x)}{|Du(x)|^2}\Big)\\[0.1cm]
&=\tr\big[|Du(x)|^{p-2}D^2u(x)\big]+(p-2) \tr\big[|Du(x)|^{p-4}Du(x)Du(x)^TD^2u(x)\big]\\[0.2cm] 
&=\tr\big[\sigma_p(Du(x))\sigma_p^T(Du(x))D^2u(x)\big]\qquad\text{where}\qquad
\sigma_p(\xi)=|\xi|^{\frac{p-2}{2}}a_p(\xi),\\[0.4cm]
H[u](x)&=\mathrm{div}\Big(\frac{Du(x)}{\sqrt{1+|Du(x)|^2}}\Big)=\frac1{\sqrt{1+|Du(x)|^2}}\Big(\Delta
u(x)-\frac{\Delta_\infty u(x)}{1+|Du(x)|^2}\Big)\\
&=\tr
\Big(\tilde\sigma(Du(x))\tilde\sigma^T(Du(x))D^2u(x)\Big)\qquad\text{where}
\qquad  
\tilde\sigma(\xi)= \frac{\tilde a(\xi)}{(1+|\xi|^2)^{\frac14}},
\end{align*}
where $\tilde \sigma$ and $\sigma_p$ are continuous for $p\geq 2$. 

\subsubsection*{First type of  examples:}
Quasilinear versions of every
generator of pure jump L\'{e}vy processes
\cite{A:Book}. E.g. nonlocal fractional Laplace type operators,
\begin{align*}
    &\mathcal{L}^{\alp/2}_{\Delta_\infty}[u](x)\quad\text{already defined in \eqref{inf-lap},}\\
%=\int_{\R^1} u\big(x+Du(x)z)-u(x)
% -Du(x)\cdot Du(x)z\,\ind{|z|<1}\frac{c_\alp dz}{|z|^{1+\alp}},\\
%\,\to\Delta_\infty u(x),
&\mathcal{L}^{\alp/2}_{\Delta_p}[u](x)=\int_{\R^N} u\big(x+\sigma_p(Du(x))z)-u(x)
-\sigma_p(Du(x))z\cdot Du(x)\,\ind{|z|<1}\frac{c_\alp \dz}{|z|^{N+\alp}},\\
& \tilde{\mathcal{L}}^{\alp/2}_{\Delta_p}[u](x)=\int_{\R^N} u\big(x+|Du(x)|^{\frac{p-2}2}z)-u(x)
-|Du(x)|^{\frac{p-2}2}z\cdot Du(x)\,\ind{|z|<1}\frac{c_\alp \dz}{|z|^{N+\alp}}\\
&\qquad\qquad+(p-2)\int_{\R^1} u\Big(x+|Du(x)|^{\frac{p-4}2}Du(x)z\Big)-u(x)
-|Du(x)|^{\frac{p-4}2}Du(x)z\cdot
Du(x)\,\ind{|z|<1}\frac{c_\alp \dz}{|z|^{1+\alp}},\\ 
&\mathcal{L}^{\alp/2}_H[u](x)=\int_{\R^N} u\Big(x+\tilde\sigma(Du(x))z\Big)-u(x)
-\tilde\sigma(Du(x)) z\cdot Du(x)\ind{|z|<1}\frac{c_\alp \dz}{|z|^{N+\alp}}
\end{align*}
where $p\geq2$ and $c_\alp=O(2-\alp)$ is the constant of $\Delta^{\alp/2}$.
The fractional Laplacian is the generator of the symmetric
$\alpha$-stable process, and the above nonlocal versions can
be seen ``generators'' of gradient dependent modulations of this process. To be
more precise, $\mathcal{L}^{\alp/2}_{\Delta_\infty}$ is a nonlocal version of
the infinity Laplace operator; both $\mathcal{L}^{\alp/2}_{\Delta_p}$ and
$\tilde{\mathcal{L}}^{\alp/2}_{\Delta_p}$ are nonlocal versions of the
$p$-Laplace operator, depending on how we write it; finally
$\mathcal{L}^{\alp/2}_{H}$ is a nonlocal version of the curvature operator $H$.
Note that these operators are of the form
$L=L_1+L_2$ where both $L_1\neq0$ ($|z|<1$) and $L_2\neq 0$
($|z|>1$).

\subsubsection*{Second type of examples:} Quasilinear versions of the
generators of some L\'{e}vy-Ito jump-processes defined by
stochastic differential equations (SDEs) driven by pure jump L\'{e}vy
processes \cite{A:Book,CT:Book}.
%:
%\begin{align*}
%X_t&=x+\int_0^t\int_{|z|>0}\tilde j_1(X_{s^-},z) \tilde N_1(\dz,\dt) +
%\int_0^t\int_{|z|>0}\tilde j_2(X_{s^-},z) N_2(\dz,\ds).
%\end{align*}
%
An example is the operator from the CGMY model for the price of a
European option in Finance \cite{CT:Book},
\begin{align*}
&\mathcal L u(x)=\int_{\R^1} u(x+z)-u(x)
-Du(x)(e^z-1)\,\frac{C e^{-Mz^+-Gz^-} \dz}{|z|^{1+Y}},
\end{align*}
for $C,G,M>0$, $Y\in(0,2)$, and the following new nonlocal infinity
Laplacian (compare to $\mathcal{L}^{\alpha/2}_{\Delta_\infty}$):
\begin{align*}
    &\mathcal{J}^{\alp/2}_{\Delta_\infty}[u](x)=\int_{\R^1} u\big(x+Du(x)z\big)-u(x)
-Du(x)\cdot Du(x)(e^z-1)\,\frac{c_\alp e^{-Mz^+-Gz^-} \dz}{|z|^{1+\alp}}.
\end{align*}
In this case $L_1\neq 0$ and
$L_2=0$, and $L=\mathcal{J}^{\alp/2}_{\Delta_\infty}$ is a gradient dependent modulation of
$\mathcal L$. Here $\mathcal L$ is not the generator of a L\'{e}vy
process, but the exponential of a L\'{e}vy process  \cite{CT:Book}
(after a transformation). The driving (L\'{e}vy) process here is a
tempered $\alp$-stable process \cite{CT:Book}. Other quasilinear versions
($p$-Laplace etc.) can be easily be constructed as above.

\begin{remark}
Since we do not allow for $x$-dependence in $j_1$ and
$j_2$ at the level of the PDE \eqref{eq:0}, we can only consider
generators of very special SDEs. In the example above the coefficients
in the SDE will depend on $X_t$, but after a change of variables this
dependence is lost in the corresponding PDE.
\end{remark}

\subsubsection*{Third type of examples:} Versions of the above
nonlocal operators with truncated and hence non-singular measures. Simply
replace $\dmu(z)$ in the definition of $L$ by $\ind{|z|>r}\dmu(z)$, e.g.
\begin{align*}
    & \mathcal{L}^{\alp/2,r}_{\Delta_\infty}[u](x)=\int_{\R^1} \Big(u\big(x+Du(x)z\big)-u(x)
-Du(x)\cdot Du(x)(e^z-1)\Big)\,\ind{|z|>r}\frac{c_\alp g(z) \dz}{|z|^{1+\alp}},
\end{align*}
where $g(0)\neq0$ and $g$ is $C^1$ at $z=0$. Note that here
$L=\mathcal{L}^{\alp/2,r}_{\Delta_\infty}$ with $L_1=0$ and $L_2\neq0$, and \hyp{\Meps}
holds with $A_2=g(0)I$ and $a=Dg(0)$.

%-------------------------------------------
\subsection{Remarks.}\

\noindent $(a)$ \brak{Continuity in $\alp,p$} All the operators above will be
continuous in 
$(\alp,p)\in(0,2)\times [2,\infty)$. For example for any bounded $C^2$
function $\phi$ and sequence $(\alp',p')\to (\alp,p)\in (0,2) \times[2,\infty)$,
$$\mathcal{L}^{\alp'/2}_{\Delta_{p'}}\phi \longrightarrow
\mathcal{L}^{\alp/2}_{\Delta_p}\phi \quad \text{in}\quad
\R^N.$$

\noindent $(b)$ \brak{The limit $\alp\to 2$} If $u$ is smooth and bounded, then
by easy computations, 
\begin{gather*}
    \mathcal{L}^{\alp/2}_{\Delta_\infty}[u]\,,\ \mathcal{J}^{\alp/2}_{\Delta_\infty}[u]
    \longrightarrow \Delta_\infty u,\qquad
    \mathcal{L}^{\alp/2}_{\Delta_p}[u]\,,\ \tilde{\mathcal{L}}^{\alp/2}_{\Delta_p}[u]
    \longrightarrow \Delta_p u,\qquad \mathcal{L}^{\alp/2}_H[u]\longrightarrow
      H [u],\\[0.2cm]
    \text{and}\qquad \mathcal{L}^{\alp/2,2-\alpha}_{\Delta_\infty}[u]
    \longrightarrow g(0)\Delta_\infty u+Dg(0)Du
\end{gather*}
point-wise as $\alp\to2$. Hence {\em all} of these operators converge
to their local counterparts including the truncated ones. These latter
operators also give rise to a drift term (when $\mu$ is
non-symmetric!). Note that in these examples assumption 
\hyp{\Meps} hold  with $\alp=2-\eps$, $A=I$ or $A=g(0)I$, and $a=0$ or
$a=Dg(0)$. 

\noindent $(c)$ \brak{Growth assumptions}
Our assumptions allow for extreme growth in
the gradient and nonlocal terms. Our results cover the equation 
$$u-F\big(L[u,Du](x)\big)=f(x)$$
for any continuous nondecreasing function $F$ and any good operator $L$ as above, e.g.
$$u-\Big(e^{\mathcal{L}^{\alp/2}_{\Delta_p}[u]}-1\Big)=f(x)
\qquad \text{for any}\qquad p\geq 2. $$ 

% \noindent (d) {\bf Move to the introduction?} (On the structure of our operators) Our nonlocal
% quasilinear operators do not arise from mimization of functionals in
% calculus of variations, as opposed to the operators defined in
% \cite{LL2014,DCKP2013,L2014} and \cite{AMRTBook}. 
% The first three paper considers minimization of fractional
% Sobolev norms ($W^{p,\frac\alp2}$-norms) while the last reference
% essentially considers minimization of truncatated versions of such norms.
% In \cite{IN2010}, a differnt ``varaiational'' type of nonlocal
% operators is studied by  non-variational viscosity solution
% techniques. None of the ``variational'' operators above have the
% implicit structure our operators do. Of existing work, our operators
% resemble most the non-variational $\infty$-Laplace type operator
% defined in \cite{BCF2012_2,BCF2012}, especially
% \cite{BCF2012_2}. Whereas the operators in 
% \cite{BCF2012_2} are have bounded dependence on the gradient but are
% discontinuous where it is zero, our operators are continuous but may
% have arbitrary growth in the gradient. The operators in
% \cite{BCF2012_2} correspond 
% to a so-called normalized $\infty$-Laplacian and are connected to a
% sequence of Tug of War games, while our operators would be unnormalized
% versions connected to an implicit control problem.

%IN2010,BCF2012,BCF2012_2,LL2014,DCKP2013,L2014

%---------------------------------------
\section{Viscosity solutions}
\label{sect:visco}
%---------------------------------------

In this section, we introduce the good notion of weak solution for
equation \eqref{eq:0}. We prove that we have two equivalent definitions
and that the solution concept is stable with respect to pointwise limits
of uniformly bounded solutions.

We start by splitting $L_1$ in \eqref{I-def} into two parts:
$L_1=L_{\delta}+L^\delta$ for $\delta>0$, where 
\begin{equation}\label{def:L.delta}
  \begin{aligned}
    L_{\delta}[\phi & ,D\phi] (x) :=  \int_{|z|<\delta} 
	\phi\big(x+j_1(D\phi(x),z)\big)-\phi(x)- 
	j_1\big(D\phi(x),z\big)\cdot D\phi(x)\ \dmu_1(z)\,,\\[2mm]
	L^{\delta}[u & ,p](x):= \int_{|z|\geq\delta} 
	u\big(x+j_1(p,z)\big)-u(x)- 
        j_1\big(p,z\big)\cdot p\ \dmu_1(z)\quad (p\in\R^N)\,.
  \end{aligned}
\end{equation}
In view of \hyp{M}, $L_\delta$ is well-defined for any $C^2$
function $\phi$ and $L^\delta$ for any bounded function $u$. 
Likewise,
the operator $L_2[u,p]$ is also well-defined for any $p\in\R^N$ and
bounded measurable function $u$. 
Recall that those integrals are taken over $\R^P$.
Now we can 
introduce the 
concept of solutions that we will use in this paper.

\begin{definition}\label{def1}\ \\
    \noindent$(a)$ A bounded usc function $u$ is a {\em viscosity subsolution} of 
  \eqref{eq:0} if for any $\delta>0$, any $C^2$ function 
  $\phi$, and any global maximum point $x$ of $u-\phi$,
  \begin{equation}\label{ineq:subsol.nolocal}
   F\Big(u(x),D\phi(x),L_{\delta}[\phi,D\phi](x)+L^{\delta}[u,D\phi](x)+L_2[u,D\phi](x)\Big)\leq f(x)\,.
  \end{equation}
  \noindent$(b)$ A bounded lsc function $u$ is a {\em viscosity supersolution} of 
  \eqref{eq:0} if for any $\delta>0$, any $C^2$ function 
  $\phi$, and any global minimum point $x$ of $u-\phi$,
  \begin{equation}\label{ineq:subsol.nolocal_1}
   F\Big(u(x),D\phi(x),L_{\delta}[\phi,D\phi](x)+L^{\delta}[u,D\phi](x)+L_2[u,D\phi](x)\Big)\geq f(x)\,.
  \end{equation}
  \noindent$(c)$ A {\em viscosity solution} is a bounded continuous function $u$ which is both a subsolution and a supersolution.
\end{definition}

Another possible definition is the following:
\begin{definition}\label{def2}\ \\
    \noindent$(a)$ A bounded usc function $u$ is a {\em
    viscosity subsolution} of  
  \eqref{eq:0} if for any bounded $C^2$ function 
  $\phi$, and any global maximum point $x$ of $u-\phi$,
  \begin{equation}\label{ineq:subsol.nolocal_2}
   F\Big(u(x),D\phi(x),L[\phi,D\phi](x)\Big)\leq f(x)\,.
  \end{equation}
  \noindent$(b)$ A bounded lsc function $u$ is a {\em
    viscosity subsolution} of  
  \eqref{eq:0} if for any bounded $C^2$ function 
  $\phi$, and any global minimum point $x$ of $u-\phi$,
  \begin{equation}\label{ineq:subsol.nolocal_3}
   F\Big(u(x),D\phi(x),L[\phi,D\phi](x)\Big)\geq f(x)\,.
  \end{equation}
  \noindent$(c)$ A {\em viscosity solution} is a bounded continuous function $u$ which is both a subsolution and a supersolution.
\end{definition}

\begin{remark} 
\label{rem:str_max}
We may assume without loss of generality that the extrema of $u-\phi$ are strict
and that $\phi=u$ at the extremal point. The latter comes from shifting the test
function by a constant. To make an extremum (say a maximum) point $x$ strict, we
replace $\phi$ by $\phi+\delta\psi$ where $\delta>0$ and
$$\psi\in C^2(\R^N)\cap W^{2,\infty}(\R^N),\qquad\psi=0 \text{ and }
D\psi=0 \text{ at } x, \qquad \text{and}\qquad \psi>0 \text{
  elsewhere,}$$
and send $\delta\to0$ in the final step of the proof.
 As opposed to the local case, the $\delta$-terms
will now be visible throughout the computations and vanish only
in the final step.
\end{remark}

\begin{lemma}
If \hyp{M}, \hyp{J1}, \hyp{F1}, and \hyp{F4} hold, then Definitions \ref{def1}
and \ref{def2} are equivalent. 
\end{lemma}
\begin{proof}
The proof is pretty standard \cite{sa,BBP,JK,BI}.
Since $(u-\phi)$ has a max in $x$, $L^\delta[u,D\phi](x)\leq
L^\delta[\phi,D\phi](x)$ and $L_2[u,D\phi](x)\leq
L_2[\phi,D\phi](x)$, and hence by \hyp{F1}, \hyp{F4}, and since
$L=L_\delta+L^\delta+L_2$, inequality 
\eqref{ineq:subsol.nolocal_2} follows from
\eqref{ineq:subsol.nolocal}. 
Conversely, we may assume the max is strict (see Remark
\ref{rem:str_max}). Then there exists 
a smooth and uniformly bounded function $\phi_\eps$ such that
$u\leq\phi_\eps\leq\phi$ and $\phi_\eps\to u$ a.e. as $\eps\to0$. It
immediately follows that also $u-\phi_\eps$ and $\phi_\eps-\phi$ 
have 
%strict 
maximum points at $x$. Hence, since $D\phi_\eps(x)=D\phi(x)$ and by the
definition of $L$ (monotonicity and  
$L=L_\delta+L^\delta+L_2$), 
\begin{align*}
%&L_\delta[\phi_\eps,D\phi](x)+L^\delta[u,D\phi](x)+L_2[u,D\phi](x)\\
%&\leq
 L[\phi_\eps,D\phi_\eps](x)\leq L_\delta[\phi,D\phi](x)+L^\delta[\phi_\eps,D\phi](x)+L_2[\phi_\eps,D\phi](x).
\end{align*}
Hence, by inequality \eqref{ineq:subsol.nolocal_2} with $\phi_\eps$
replacing $\phi$, inequality \eqref{ineq:subsol.nolocal} with $\phi_\eps$
replacing $u$ follows. Now we conclude by sending $\eps\to0$, using
\hyp{M}, \hyp{J1}, \hyp{F1}, \hyp{F4}, and the dominated convergence theorem.
\end{proof}

Next, we show that this solution concept is stable with respect
to local uniform limits, to so-called half-relaxed limits, and more
generally to very general perturbations of the equation. Consider
\begin{equation}\label{eq:eps2}
    F_\eps\big(u_\eps,Du_\eps,L_\eps[u_\eps,Du_\eps]\big)=f_\eps(x)\qquad\text{in}\qquad \R^N,
\end{equation}
where  $(f_\eps,F_\eps,L_\eps:=L_{1,\eps}+L_{2,\eps})$ satisfy
\hyp{F4}, \hyp{F1}, \hyp{M}, and \hyp{J1} for each
fixed $\eps>0$, and where
\begin{align*}
%\label{I-def2}
L_{1,\eps}[u,Du](x)&=\int_{\R^P} u\big(x+j_{1,\eps}(Du,z)\big)-u(x)
-j_{1,\eps}(Du,z)\cdot Du(x)\, \dmu_{1,\eps}(z)\,,\\
L_{2,\eps}[u,Du](x)&=\int_{\R^P} u\big(x+j_{2,\eps}(Du,z)\big)-u(x)
\dmu_{2,\eps}(z)\,.
\end{align*}
Then we define the ``half-relaxed limits'':
\begin{align}
& \ol u(x):=\limsup_{y\to x,\eps\to 0} u_\eps(y)\quad 
\text{and}\quad \displaystyle \ul u(x):=\liminf_{y\to x,\eps\to 0} u_\eps(y).\\
& \ol f(x):=\limsup_{y\to x,\eps\to 0} f_\eps(y)\quad
\text{and}\quad  \ul f(x):=\liminf_{y\to x,\eps\to 0} f_\eps(y).\\
& \ol F(u,p,l):=\!\!\!\limsup_{\scriptsize\begin{array}{c}{(v,q,m)\to (u,p,l)}\\{\eps\to 0}\end{array}}\hspace{-0.5cm}F_\eps(v,q,m)\quad \text{and}\quad \ul F(u,p,l):=\!\!\liminf_{\scriptsize\begin{array}{c}{(v,q,m)\to (u,p,l)}\\{\eps\to 0}\end{array}}\hspace{-0.5cm}F_\eps(v,q,m).
\end{align}

\begin{lemma}[Stability 1] \label{lem:stab} Assume $\{f_\eps,F_\eps,L_\eps\}_\eps$ satisfy \hyp{M},
  \hyp{J1}, \hyp{F1}, \hyp{F4} for any $\eps>0$, 
$$\liminf_{\eps\to0}L_\eps[\phi,D\phi](x_\eps)\leq L[\phi,D\phi](x)
\qquad (\text{resp.}\  \limsup_{\eps\to0}L_\eps[\phi,D\phi](x_\eps)\geq L[\phi,D\phi](x)),$$
for all bounded $\phi\in C^2$ and all sequences $x_\eps\to x$, and that
$\{u_\eps\}_{\eps>0}$ is a sequence of uniformly 
bounded subsolutions (resp. supersolutions)  of \eqref{eq:eps2}. 

Then $\displaystyle\ol u(x)$ is a subsolution (resp. $\displaystyle\ul
u(x)$ supersolution)
of \eqref{eq:0} with $(f,F)$ replaced by $({\overline f},{\underline
  F})$ (resp. $({\underline f},{\overline
  F}))$. 
\end{lemma}

We also have the following stability result.

\begin{lemma}[Stability 2]\label{lem:stab2}%[Local uniform convergence]
Assume \hyp{M}, \hyp{J1}, \hyp{F1}, \hyp{F4} hold, and $\{u_a\}_{a\in
  A}$, for a set $A$, is a family of uniformly 
bounded subsolutions (resp. supersolutions) of \eqref{eq:0}.  

\noindent $(a)$ If for any $n$, $u_{a_n}$ is continuous and $u_{a_n}\to u$ locally
uniformly as $n\to\infty$, then $u$ is a continuous bounded
subsolution (resp. supersolution) of \eqref{eq:0}.

\noindent $(b)$ $u:=\sup_{a\in A} u_a$ is a subsolution (resp. $v=\inf_{a\in
  A}u_a$ is a supersolution) of \eqref{eq:0}. 
\end{lemma}

The proofs follow after the next remark.

\begin{remark}
    $(i)$ Similar type of results can be found in \cite{BI}, but
  without variation in $L$.

\noindent $(ii)$ Compare Lemma \ref{lem:stab2}-$(a)$ to the no stability
w.r.t. local uniform convergence result of \cite{BCF2012_2}. In
\cite{BCF2012} there is stability, but the nonlocal 
  operators are more different from ours than in \cite{BCF2012_2}.  
\end{remark}

\begin{proof}[Proof of Lemma \ref{lem:stab}]
The proof is quite standard, see e.g. \cite{BI} (Theorem 2) for a 
similar proof. We only do the subsolution case since the supersolution case
is similar. Assume $\phi$ is  $C^2$ and bounded and $\bar
u-\phi$ has a global maximum at $x$, we will show that inequality
\eqref{ineq:subsol.nolocal_2} holds and we are done.

Modifying the test function if necessary (as in Remark \ref{rem:str_max},
assuming also $\psi(y)=1$ for $|y|>1$), we may assume the
maximum is unique, strict, and can not be attained at infinity. In
fact, we may assume that
\begin{align}
\label{qq} (\bar u-\phi)(x)>\sup_{|y-x|>r}(\bar
u-\phi)(y)\qquad\text{for any}\quad r>0.
\end{align}
Then we take a subsequence such that $\bar u (x) = 
\lim_{\eps} u_{\eps} (x_{\eps})$, and note that by \eqref{qq} and
classical arguments \cite[Lemma V.1.6]{CB:Book}, we may
find a sequence $\{y_\eps\}_\eps$ such that
$$u_{\eps}-\phi \quad\text{has a global maximum at
  $y_{\eps}$},\qquad  y_{\eps}\to x, \qquad\text{and}\qquad
u_{\eps}(y_{\eps})\to \overline u(x).$$
Since $u_{\eps}$ is a subsolution of \eqref{eq:eps2},
$$
F_{\eps} (u_{\eps} (y_{\eps}),
D\phi(y_{\eps}),l_{\eps})\le f_\eps(y_{\eps})\qquad\text{where}\qquad
l_{\eps}=L_{\eps}[\phi,D\phi](y_{\eps}). 
$$
By the construction of $y_\eps$ and the assumption of the Lemma,
\begin{align}
\label{qqq}
\liminf_{\eps}\, l_{\eps} \leq L[\phi,D\phi](x).
\end{align}
Hence if we take a further subsequence in $\eps$ such that
$l_\eps\to\liminf_{\eps}\, l_{\eps}$, then by the definition of
$(\overline f,\underline F)$ and continuity, \hyp{F1}
and \hyp{F4},
$$
\underline F (\bar u(x),D\phi(x),\liminf_{\eps}\, l_{\eps})\leq \overline f(x).
$$ 
and inequality \eqref{ineq:subsol.nolocal_2} then follows from
\eqref{qqq} and monotonicity \hyp{F1}.
\end{proof}

\begin{proof}[Proof of Lemma \ref{lem:stab2}]
$(a)$ Let $\phi$ be $C^2$ and bounded and $x_\eps\to x$. By assumptions \hyp{M}
and \hyp{J1}, and the dominated convergence theorem, 
$$
\lim_{x_\eps\to x}L[\phi,D\phi](x_\eps)=L[\phi,D\phi](x).
$$
Hence by Lemma \ref{lem:stab}, $\overline
u(x)=\displaystyle\limsup_{y\to x,n\to\infty} u_{a_n}(y)$ is a (bounded) subsolution
of \eqref{eq:0}. Since $u_{a_n}$ is continuous and $u_{a_n}\to u$
locally uniformly, it follows that $\overline u=u$ and $u$ is continuous.

\noindent $(b)$ The proof is similar to the proof of Lemma \ref{lem:stab} and we
only do the subsolution case. Assume $u-\phi$ has a strict global max
at $x$. By the definition of the supremum, there is a sequence
$u_{a_k}(x_k)\to u(x)$ as $k\to\infty$. As in the previous proof we
may find a sequence $\{y_k\}_k$ such that such that
$$u_{a_k}-\phi \quad\text{has a global maximum at
  $y_{k}$},\qquad  y_{k}\to x, \qquad\text{and}\qquad u_{a_k}(y_{k})\to u(x).$$
Since $u_{a_k}$ is a subsolution of \eqref{eq:0},
$$
F(u_{a_k} (y_{k}),
D\phi(y_{k}),l_{k})\le f(y_{k})\qquad\text{where}\qquad
l_{k}=L[\phi,D\phi](y_{k}). 
$$
By the construction of $y_k$, assumptions \hyp{M} and \hyp{J1}, and
the dominated convergence theorem, 
$$
\lim_{k}\, l_{k}=L[\phi,D\phi](x),
$$
and then by the continuity, \hyp{F1} and \hyp{F4}, inequality
\eqref{ineq:subsol.nolocal_2} holds. 
\end{proof}

%------------------------------------------------------------------
\section{Proofs of the main results}
\label{sect:proofs}
%------------------------------------------------------------------

%-------------------------------------
\subsection{Proof of Theorem \ref{thm:comp} (comparison)}

\begin{proof}[Proof of Theorem \ref{thm:comp}-$(a)$]
We proceed by contradiction, assuming that $M:=\sup\big(u-v)>0$. 

Let $\eps,R>0$ and define
  \begin{equation}\label{Phi}
    \Phi_{\eps,R}(x,y):=u(x)-v(y)-\phi(x,y),
  \end{equation}
  where
  \begin{equation}\label{test_func}
                         \phi(x,y)=\frac1{\eps^2}\varphi(x-y)
			+\psi\left(\frac x R\right)+\psi\left(\frac y
                          R\right)\,, 
		\end{equation}
		and $\varphi,\psi$ are smooth bounded radially
                symmetric and radially non-decreasing functions such that 
$$\varphi(x)=\begin{cases}|x|^2 & \text{for } |x|<1\\ 2& \text{for } |x|>4  \end{cases}\qquad\text{and}\qquad
\psi(x)=\begin{cases}0 & \text{for }|x|<\frac12\\ 2(\|u\|_\infty+\|v\|_\infty) +
                1 &\text{for }|x|>1
\end{cases}$$

By penalization (the $\psi$-terms) the supremum of $\Phi_{\eps,R}$ is attained
at a point $(\bx,\by)$, and since $M>0$ this supremum is
positive when $R$ is big enough (see {\bf 1)} below):
$$M_{\eps,R}:=\max\Phi_{\eps,R}=\Phi(\bx,\by)>0.$$
                For the sake of simplicity we drop the
                reference to $\eps,R$ for the maximum
                point. 
% From the inequality $\Phi(0,0)\leq\Phi(\bx,\by)$ and the definition of
% $\psi$, 
% \begin{equation*}
%                     |\bx|,|\by|\leq R.
% \end{equation*}
By the inequality
                $\Phi(\bx,\bx)+\Phi(\by,\by)\leq 2\Phi(\bx,\by)$, it
                follows that $\frac2{\eps^2}\varphi(\bx-\by)\leq u(\bx)-u(\by)+v(\bx)-v(\by),$
and hence
\begin{equation}
\label{xy-estim}
\varphi(\bx-\by)\leq
(\|u\|_\infty+\|v\|_\infty)\eps^2.
\end{equation}
By taking $\eps>0$ small enough, we can 
always assume that  
$$\varphi(\bx-\by)=|\bx-\by|^2\qquad\text{and}\qquad (D\varphi)(\bx-\by)=2(\bx-\by).$$
In particular, $|\bx-\by|\leq(\|u\|_\infty+\|v\|_\infty)\eps$ and this estimate is
independent of $R$. 

From the maximum of $\Phi_{\eps,R}$ it follows that
		$u(x)-\phi(x,\by)$ has a global maximum point at $\bx$
                and $v(y)-(-\phi)(\bx,y)$ has a global minimum point at
                $\by$. Subtracting the corresponding viscosity
                inequalities  for $u$ and  $v$ (cf. Definition
                \ref{def1}) gives for any $\delta>0$ that
		\begin{align}
		    0\geq & -\bigg(L_{\delta}[\phi(\cdot,\by),D_x\phi](\bx)-
			L_{\delta}[(-\phi)(\bx,\cdot),D_y(-\phi)](\bar
                        y)\bigg)\nonumber\\
&			 -\bigg(L^{\delta}[u,D_x\phi](\bx)-
			L^{\delta}[v,D_y(-\phi)](\bar
                        y)\bigg)\nonumber \\
&-\bigg(L_2[u,D_x\phi](\bx)-
			L_2[v,D_y(-\phi)](\bar
                        y)\bigg)\nonumber\\
&-(f(\bx)-f(\by)) + \big(u(\bx)-v(\by)\big)\nonumber\\
\geq& 
			-I_{\delta}-I^{\delta} -I_2- \omega_f(\bx-\by)+
			\big(u(\bx)-v(\by)\big)\,. \label{vineq}
			\end{align}

		The strategy is now to estimate $I_{\delta}$,
                $I^\delta$, and $I_2$, and prove that when
                sending first $\delta\to0$, then $R\to\infty$, and
                finally $\eps\to 0$,
\begin{align*}
&\limsup_{\eps\to0}\limsup_{R\to\infty}\limsup_{\delta\to0}\big(I_{\delta}
+I^\delta+I_2\big)\leq0.
\end{align*}
We will also show that
\begin{align}
&\limsup_{\eps\to0}\limsup_{R\to\infty}\limsup_{\delta\to0}
\big(u(\bx)-v(\by)\big)\geq M,\label{Mlim}
\end{align}
and hence by the viscosity inequality \eqref{vineq} we get the contradiction that concludes the proof:
$$0\geq M. $$

We proceed in 4 steps:
\smallskip
				
\noindent\textbf{\bf 1)}\quad We show that \eqref{Mlim} holds. First
note that 
$u(\bx)-v(\by)=M_{\eps,R}+\phi(\bx,\by)$ does not depend on
$\delta$. Then by the maximum point property, it follows that
$$M_{\eps,R}\ra
M_\eps:=\sup
\Big(u(x)-v(y)-\frac1{\eps^2}\varphi(x-y)\Big)\qquad\text{and}\qquad
\psi(\tfrac {\bx}R)+\psi(\tfrac {\by}R)\to 0$$
 as $R\to\infty$ (see Lemma 2.3 in \cite{JK02}). Observe now that $M\leq
 M_\eps\leq M_{\eps'}$ for $\eps\leq\eps'$, and hence by monotone convergence,
$M_\eps\searrow \tilde M$ for some $\tilde M\geq M$. 
Since $\displaystyle\tilde M
=\limsup_{\eps\to0}\limsup_{R\to\infty}\limsup_{\delta\to0} 
\big(u(\bx)-v(\by)\big)$, we are done.
                \smallskip
				
		\noindent\textbf{\bf 2)}\quad To estimate the
                $I_{\delta}$-term, we  Taylor expand to find that
		\begin{equation*}
		\begin{aligned}
			\int_{|z|<\delta} 
		\phi & \big(\bx+j_1(D_x\phi(\bx,\by),z),\by\big)- \phi(\bx,\by)- 
		j_1\big(D_x\phi(\bx,\by),z\big)\cdot D_x\phi(\bx,\by)\dmu_1(z)\\
		&\leq \|D^2\phi\|_{\infty}%(B_{r'}(\bx))} 
                \int_{|z|<\delta} 
		\big|j_1(D_x\phi(\bx,\by),z)\big|^{2}\dmu_1(z)=o_\delta(1)
		\end{aligned}
		\end{equation*}
                for fixed $\eps,R>0$. Here
%                $r'=\max\{|j_1(p,z)|:|x|\leq R,|p|\leq
%                C_\eps,|z|\leq 1\}$ and 
                the $o_\delta(1)$ comes from assumption \hyp{J1} and
                dominated convergence as $\delta\to0$. After a similar
                estimate for $L_\delta[-\phi,D_y(-\phi)]$, we
                conclude that $I_{\delta}\to0$ as $\delta
                \to0$ and $\eps,R>0$ are fixed.
\smallskip

 \noindent\textbf{3)}\quad We estimate $I^\delta$. 
                Using the notation $j_{\bx}(z):=j_1(D_x\phi(\bx,\by),z)$
                and $j_{\by}(z):=j_1(D_y(-\phi)(\bx,\by),z)$, and the maximum point
                property of $\Phi_{\eps,R}$,
$$\Phi_{\eps,R}(\bx+j_{\bx},\by+j_{\by})\leq\Phi_{\eps,R}(\bx,\by),$$ 
we see that
		\begin{equation*}
			\begin{aligned}
			I^{\delta}=
			&\int_{\delta\leq |z|}
                        \bigg(\Big(u\big(\bx+j_{\bx}(z)\big)-u(\bx)\Big)-\Big(
                        v\big(\by+j_{\by}(z)\big)-v(\by)\Big) \\
                        &\qquad\qquad\quad +j_{\bx}(z)\cdot D_x\phi(\bx,\by)+ 
                        j_{\by}(z)\cdot D_y\phi(\bx,\by)\bigg)\dmu_1(z)\\
\leq &
                \int_{\delta\leq|z|}\Big( 
                \phi\big(\bx+j_{\bx}(z),\by+j_{\by}(z)\big)-\phi(\bx,\by)
                -D_x\phi(\bx,\by)\cdot j_{\bx}(z)-D_y\phi(\bx,\by)\cdot j_{\by}(z)
                \Big)\dmu_1(z).
			\end{aligned}
		\end{equation*}
                Since $D^2\varphi$ is bounded and $|D^2\psi(\tfrac x R)|\leq
                \tfrac 1 {R^2}\|D^2\psi\|_{\infty}<\infty$, a short
                computation using Taylor expansions shows that
                \begin{align*}
                    I^{\delta}  & \leq \int_{\delta\leq|z|}\Big(
                    \tfrac1{2\eps^2}\|D^2\varphi\|_\infty\big| j_{\bx}(z)-j_{\by}(z)\big|^2 +
                    \tfrac 1{2R^2}\|D^2\psi\|_\infty\big(|j_{\bx}(z)|^2+|j_{\by}(z)|^2\big)\Big)
                    \dmu_1(z)\,.
%\\
%                    &\leq\int_{\delta\leq|z|}
%                    \Big(O(\tfrac1{\eps^2})\big| j_{\bx}(z)-j_{\by}(z)\big|^2 +
%                    O(\tfrac 1{R^2})\big(|j_{\bx}(z)|^2+|j_{\by}(z)|^2\big)\Big)
%                    \dmu_1(z)\,. 
                \end{align*}
To proceed we compute the gradients, 
 		\begin{equation*}
 			D_x\phi(\bx,\by)=p_\eps + 
                         \tfrac{1}{R}D\psi(\tfrac{\bx}{R})
                        \,,\quad
 			D_y(-\phi)(\bx,\by)=p_\eps - 
                         \tfrac{1}{R}D\psi(\tfrac{\by}{R})
                         \,,\quad
                         p_\eps=\frac{2(\bx-\by)}{\eps^2}\,,
 		\end{equation*}
and note that for fixed $\eps>0$, they are uniformly bounded for
$R>1$ by estimate \eqref{xy-estim}.
Hence, there is $r_\eps>0$ such that 
                $|D\phi(\bx,\by)|\leq r_\eps$ for all $\delta>0$ and
                $R>1$, and then by assumptions \hyp{J1} and \hyp{J2}, 
                \begin{equation*}
                    \begin{aligned}
                        I^\delta &\leq O(\tfrac1{\eps^2})\, \omega_{j,r_\eps}\big(D_x\phi(\bx,\by)-
                    D_y(-\phi)(\bx,\by)\big)+
                    O(\tfrac1{R^2})C_{j,r_\eps}\\
                    &\leq O(\tfrac1{\eps^2})\,
                    \omega_{j,r_\eps}\big(O(\tfrac1{R})\big)+
                    O(\tfrac1{R^2})C_{j,r_\eps}\,.
                \end{aligned}
                \end{equation*}
                We first send $\delta\to0$ since nothing depends on $\delta$ on
                the right-hand side, and then we send $R\to\infty$ and
                find that
                $$\limsup_{R\to\infty}\limsup_{\delta\to0}
                I^\delta\leq0\,.$$

\smallskip

 \noindent\textbf{4)}\quad  Finally, we estimate $I_2$. First note that
 by the maximum point property, the positivity of $\phi$, the calculations
 of gradients in $(c)$, and estimate \eqref{xy-estim},
\begin{align*}
I_2 & \leq  \int_{|z|>0}
\phi\Big(\bx+j_2\big(p_\eps+\tfrac{1}{R}D\psi(\tfrac{\bx}{R}),z\big),\by+j_2\big(p_\eps-\tfrac{1}{R}D\psi(\tfrac{\by}{R}),z\big)\Big)-\phi(\bx,\by)\
\dmu_2(z)\\
& \leq\frac1{\eps^2}
\int_{|z|>0}\sup_{x,y\in\R^N}\Big\{\varphi\Big(x+j_2\big(p_\eps+\tfrac{1}{R}D\psi(\tfrac{x}{R}),z\big)-\big(y+j_2\big(p_\eps-\tfrac{1}{R}D\psi(\tfrac{y}{R}),z\big)\big)\Big)
\\
&\hspace{10.8cm}-\varphi(x-y)\Big\}
\dmu_2(z)\\
& \quad+ \int_{|z|>0}
\psi\Big(\frac{\bx+j_2\big(p_\eps+\tfrac{1}{R}D\psi(\frac{\bx}{R}),z\big)}{R}\Big)+\psi\Big(\frac{\by+j_2\big(p_\eps-\tfrac{1}{R}D\psi(\frac{\by}{R}),z\big)}{R}\Big)\
\dmu_2(z)\\
&:= J_1+J_2.
\end{align*}

Now we send $R\to\infty$ in $J_1$. Then 
by compactness, $p_\eps$ will up to
a subsequence  converge to a limit that we also call $p_\eps$. By
the boundedness of $D\psi$ and $p$-continuity of $j_2(p,z)$ for
a.e. $z$ in \hyp{J1}, 
$$\lim_{R\to\infty}\sup_{x\in\R^N} \Big|j_2\big(p_\eps\pm\tfrac 1 R
D\psi(\tfrac x R),z\big)-j_2\big(p_\eps,z\big)\Big|=0\quad\text{for
  a.e. $z$}.$$ 
Hence since $\varphi$ is a Lipschitz continuous function, 
\begin{align*}
&\sup_{x,y\in\R^N}\Big| \varphi\Big(x-y-j_2\big(p_\eps+\tfrac 1 R
D\psi(\tfrac x R),z\big)+j_2\big(p_\eps-\tfrac 1 R
D\psi(\tfrac y R),z\big)\Big)-\varphi(x-y)\Big|\\
&\leq \|D\varphi\|_\infty \sup_{x,y\in\R^N}\Big|j_2\big(p_\eps+\tfrac 1 R
D\psi(\tfrac x R),z\big)-j_2\big(p_\eps-\tfrac 1 R
D\psi(\tfrac y R),z\big)\Big|\ra 0 
\end{align*}
as $R\to\infty$ for a.e. $z$. Hence,  the
$J_1$-integrand is a uniformly bounded function
converging  to $0$ as $R\to\infty$  for a.e. $z$. Hence by the
dominated convergence theorem (for fixed $\eps$),
\begin{align*}\limsup_{R\to\infty}J_1 \leq 0.
\end{align*}

Now we send $R\to\infty$ in $J_2$. Here we use the fact that 
$$\psi(\tfrac\bx R)+\psi(\tfrac\by R)\to 0\qquad\text{as}\qquad R\to \infty,$$
which is a simple consequence of the maximum point property (see Lemma
2.3 in \cite{JK02}). Since $j_2$ is locally bounded for a.e. fixed $z$
by \hyp{J1} and $\psi$ is continuous,
$$\frac{j_2\big(p_\eps+\tfrac 1 R D\psi(\tfrac \by R),z\big)}{R}\to 0\quad\text{for a.a.
}z $$
as $R\to\infty$, and hence
$$\psi\Big(\frac{\bx+j_2\big(p_\eps+\tfrac 1 R D\psi(\tfrac \bx
  R),z\big)}{R}\Big)\to \psi(0)=0\quad\text{for a.a. }z $$
as $R\to\infty$. Since $\psi$ is bounded, we can use
the dominated convergence theorem to conclude that
$$\limsup_{R\to\infty}J_2=0. $$
Since $I_2$ is independent of $\delta$, we can now conclude that
                $$
                  \limsup_{R\to\infty}\limsup_{\delta\to 0}I_2\leq0\,,
                $$
and the proof is complete.
	\end{proof}

        \begin{proof}[Proof of Theorem \ref{thm:comp}-$(b)$]
Part of the proof is similar to the previous proof.
    We start by assuming that $M:=\sup\{u(x)-v(x)\}>0$ and consider
    the maximum $M_{\eps,R}$ of  
    $$\Phi_{\eps,R}(x,y):=u(x)-v(y)-\phi(x,y)\,,$$
    where $\phi$ was defined in the proof of Theorem~\ref{thm:comp}-$(a)$.
    Since $M_{\eps,R}\to M_\eps\geq M$ as $R\to\infty$, we may assume that $M_{\eps,R}\geq M/2$ and $u(\bx)>v(\by)$. Since
    $u-\phi(\cdot,\by)$ has a 
    global max in $\bx$ and $v-(-\phi)(\bx,\cdot)$ has a
    global min in $\by$, we subtract the corresponding viscosity
    inequalities and find that
    \begin{align}\label{diff}
    & F\Big(u(\bx),p_{\eps}+O(\tfrac1 R),\underset{I_x}{\underbrace{L_\delta[\phi(\cdot,\by),D_x\phi](\bx)
    +L^\delta[u,D_x\phi](\bx)+L_2[u,D_x\phi](\bx)}}\Big)\nonumber\\
    &- 
    F\Big(v(\by),p_{\eps}+O(\tfrac 1 R),\underset{I_y}{\underbrace{L_\delta[-\phi(\bx,\cdot),D_y(-\phi)](\by) 
    +L^\delta[v,D_y(-\phi)](\by)+L_2[v,D_y(-\phi)](\by)}}\Big)\\
&\leq f(\bx)-f(\bar y)\,,\nonumber
    \end{align}
where $p_\eps=2(\bx-\by)/\eps$.

We now estimate the different terms. By the estimates in the proof of
Theorem~\ref{thm:comp}-$(a)$,  
\begin{align*}
      &
      \Big|L_\delta[\phi(\cdot,\by),D_x\phi](\bx)-L_\delta[\phi(\bx,\cdot),D_y(-\phi)](\by)\Big| 
      =\frac1{\eps^2} o_\delta(1)\,,\nonumber\\
      & \Big|L^\delta[u,D_x\phi](\bx)-L^\delta[v,D_y(-\phi)](\by)\Big|+\Big|L_2[u,D_x\phi](\bx)-L_2[v,D_y(-\phi)](\by)\Big|
      =\Big(1+\frac1{\eps^2}\Big)o_R(1)\,,\\
\intertext{and hence}
&|I_x-I_y|= \Big(1+\frac1{\eps^2}\Big)(o_R(1)+o_\delta(1)).%\label{L-diff}
    \end{align*}
By the order we will take the limits, we may and will always
assume that terms on the right hand sides are bounded (by $1$ for example).
Moreover, by the estimates in the proof of Theorem~\ref{thm:comp}-$(a)$,
\begin{align*}
&|D\phi(\bx,\by)|\leq C\Big(\frac1\eps+\frac1R\Big),\\
&|D^2\phi(\bx,\by)|\leq C\Big(\frac1{\eps^2}+\frac1{R^2}\Big),\\
&\big|L_\delta[\phi(\cdot,\by),D_x\phi](\bx)\big|\leq 
\|D^2\phi\|_\infty\int_{|z|<\delta}|j_1(D_x\phi(\bx,\by),z)|^2\dmu_1(z)\,,\\ 
&\big|L_\delta[-\phi(\bx,\cdot),D_y(-\phi)](\by)\big|\leq 
\|D^2\phi\|_\infty\int_{|z|<\delta}|j_1(D_y(-\phi)(\bx,\by),z)|^2\dmu_1(z)\,,\\ 
& \big|L_2[u,D_x\phi](\bx)\big|+\big|L_2[v,D_y(-\phi)](\by)\big|
\leq 2\big(\|u\|_\infty\vee\|v\|_\infty\big)\, \mu_2(\R^N),
\end{align*}
and by the maximum point property,
\begin{align*}
&L^\delta[u,D_x\phi](\bx) \leq
L^\delta[\phi(\cdot,\by),D_x\phi(\bx,\by)](\bx)\\
& \leq \|D^2_x\phi\|_\infty\int_{|z|>0}|j(D_x\phi(\bx,\by),z)|^2\dmu_1(z)
-\|D^2_y\phi\|_\infty\int_{|z|>0}|j(D_y(-\phi)(\bx,\by),z)|^2\dmu_1(z)\\
&\leq L^\delta[v,D_y(-\phi)](\by) \,.
\end{align*}
If $\eps>0$ is fixed, then by \hyp{M} and \hyp{J1}, these terms are
uniformly bounded for $R>1$ and $\delta>0$.  
These and the previous bounds then implies that 
there is a $C_\eps>0$ such that
$$-C_\eps\leq  L^\delta[v,D_y(-\phi)](\by) \leq
L^\delta[u,D_x\phi](\bx)+\big|L^\delta[v,D_y(-\phi)](\by)-L^\delta[u,D_x\phi](\bx)\big|\leq C_\eps, $$
and similarly we can show that $|L^\delta[u,D_x\phi](\bx)|\leq
C_\eps$. Hence there is $r_\eps>0$ such that
$$|D\phi(\bx,\by)|+|I_x|+|I_y|\leq
r_\eps\qquad\text{for all}\qquad R>1,\ \delta>0.$$

 By \eqref{diff} and the previous estimates, \hyp{F2}, \hyp{F3},
 \hyp{F4}, we see 
    $$\begin{aligned}
        \gamma\frac{M}{2} \leq & \ 
        F\big(u(\bx),p_\eps-O(\tfrac 1 R),I_y\big) -
        F\big(u(\bx),p_\eps+O(\tfrac 1 R),I_x\big)+f(\bx)-f(\by)\\ 
       \leq & \ \omega_{\|u\|_\infty \vee\|v\|_\infty,r_\eps}\Big(O(\tfrac 1
       R)+\Big(1+\frac1{\eps^2}\Big)(o_R(1)+o_\delta(1))\Big)+o_\eps(1).\\   
     \end{aligned}    
    $$
Sending first $\delta\to0$, then $R\to\infty$, and finally $\eps\to0$,
 we get again $M\leq0$. This is a contradiction and the result follows.
\end{proof}

%---------------------------------------------------------------------
\subsection{Proof of Theorem \ref{thm:existence} (existence)}
\label{sect:existence}

A major challenge we face when we want to prove existence, is the
implicit nature of equation \eqref{eq:0} with a gradient dependence
inside the $j$ functions. It seems
non-trivial to use 
Perron's method for such equations, and since fixed point iterations require
convergence of the full sequence to get the equation in the limit,
compactness argument (yielding subsequences) can not work. We have been
able to overcome the problem by a nontrivial approximation procedure, which is
inspired by the ``Sirtaki method'' of \cite{BCGJ}, along with a
fixed point argument using Schauder's fixed point theorem. We start by
proving existence for an approximate problem in a bounded set, and
then pass to the limit using the method of half-relaxed limits and strong
comparison of the limit equation.

We begin with the linear case \eqref{eq:main.nonlocal}. The simple
adaptations for the general case are given at the end of this section. 
Consider now the following approximate problem: find $u\in C^2(\overline B_R)$ 
such that
\begin{equation}\label{approx.pb}
 \begin{cases}   u-T_M\Big[L^R_{k}[u,Du]\Big]-\eps\Delta u=f(x)\,,&\quad
        x\in B_R(0)\,,\\
u=0,&  \quad       x\in \partial B_R(0)\,,
\end{cases}
\end{equation}
where 
$L^R_{k}[v,Dv]=L^R_{1,k}[v,Dv]+L^R_{2,k}[v,Dv]$ for
\begin{align*}
L^R_{1,k}[v,Dv](x)&=\int_{\R^P}v
\Big(P_R\big(x+j_{1}(Dv(x),z)\big)\Big)-v(x)-Dv(x)\cdot
j_{1}(Dv(x),z)\ind{|z|<1}\ \d\mu_{1,k}(z)\,,\\
L^R_{2,k}[v,Dv](x)&=\int_{\R^P}v
\Big(P_R\big(x+j_{2}(Dv(x),z)\big)\Big)-v(x)\ \d\mu_{2,k}(z)\,,
\end{align*}
$T_M$ is a truncation and $P_R$ the orthogonal projection onto $\ol B_R$, 
$$T_M[f]:=\min(\max(f,-M),M),\qquad%\text{and}
\qquad P_R(x):=\begin{cases}x & \text{if } |x|\leq R,\\[0.2cm]
\dfrac{R}{|x|}x& \text{if } |x|>R,\end{cases}$$
and the measures
$$\mu_{1,k}:=\rho_k\ast(\mu_1\cdot
\ind{1/k<|z|<k})\qquad\text{and}\qquad\mu_{2,k}:=\rho_k\ast \mu_2,$$ 
for a mollifier
$\rho_k(z)=k^P\rho(kz)$, $0\leq\rho\in C^\infty(\R^P)$ is
symmetric with support in $B_1$ and $\int\rho=1$.

\begin{remark}\label{remP}
\noindent $(i)$ The truncated mollified measures $\mu_{1,k}$ and $\mu_{2,k}$ are
absolutely continuous with respect to the Lebesgue measure with
bounded densities, see Lemma \ref{lem:measure} below.

\noindent $(ii)$ From the definition it follows that
$$|P_R(x)-P_R(y)|\leq|x-y|\qquad \text{for all}\qquad x,y\in\R^N.$$
The projection allows us to look for
solutions that are defined only in $\overline B_R$ and not in all of
$\R^N$ as in equation \eqref{eq:0}. The new nonlocal term is of
Neumann-type, corresponding to jump processes that are projected back to the
boundary of the domain immediately upon leaving it (cf.\cite{BCGJ,BGJ}). 
%  Other choices choices are also
% possible, but may not have the same good analytical properties
% (e.g. we could not use censored processes since this model is less regular)
\end{remark}

To prove existence we first strengthen the assumptions on $j$ and
$f$, later we do the general case.

\noindent\hyp{J1'} $j_1(p,z)$  and $j_2(p,z)$ are Borel measurable,
locally bounded, continuous in $p$ for a.e. $z$, and for every $r>0$
there is $C_r$ such that for all
$|p|\leq r$ and $|z|<1$, $$|j_1(p,z)|\leq C_{r}|z|$$

\noindent\hyp{J2'} for any $K>0$, there exists $C=C(K)$ such that for any
$|p|,|q|,|z|\leq K$, 
$$    \big|j_{i}(p,z)-j_{i}(q,z)\big|\leq C|p-q|\,.$$

\noindent \hyp{F5'} $f:\R^N\to\R$ is bounded and Lipschitz continuous.

\begin{remark}
\label{rem:newass}
For any $x\in B_R$ and $u\in C^2(\ol B_R)$, all terms
in \eqref{approx.pb} are well-defined and the equation holds in the
classical sense (since $\mu_{i,k}$ are bounded and $P_R$ is continuous,
the integral terms are well-defined because of \hyp{J1'}).
 Note that \hyp{J1'} and \hyp{M} implies \hyp{J1}, while $\mu_1$
bounded and compactly supported and \hyp{J2'} implies \hyp{J2}. 
\end{remark}

We state the properties of $\mu_{i,k}$ that we will need later. The
 proof is given in Appendix \ref{sec:pf1}.  
\begin{lemma}\label{lem:measure}
    Assume \hyp{M} holds and $\delta>0$.
    \\[2mm]
    $(a)$ The measures $\mu_{1,k}$ and $\mu_{2,k}$ have densities
    $\bar\mu_{1,k}$ and $\bar\mu_{2,k}$ with
    respect to the Lebesgue measure on $\R^N$ such that
$$\|\bar \mu_{1,k}\|_\infty,\|\bar \mu_{2,k}\|_\infty<\infty\qquad\text{and}\qquad \int_{|z|<\delta}|z|^2\bar\mu_{1,k}(z)\,dz\leq
 4 \int_{|z|<\delta}|z|^2\,\mu_{1}(dz).$$
From now on, let $(\psi_k)$ be a sequence of functions and $C>0$
constants independent of $p,z$ that differ from line to 
line.    \\[3mm] 
    $(b)$ If $\|\psi_k\|_\infty\leq C$ for all $k$, and
    $\displaystyle\sup_{\delta\leq|z|<K}|\psi_k(z)-\psi(z)|\underset{k\to\infty}{\longrightarrow}0$ for any $K>\delta$,
    then 
    $$\int_{|z|\geq\delta}\psi_k(z)\,\d\mu_{1,k}(z)\underset{k\to\infty}{\longrightarrow}
    \int_{|z|\geq\delta}\psi(z)\,\mu_{1}(\!\dz)\,.$$
\\[2mm]
    $(c)$ If $|\psi_k(z)|\leq C|z|^2$ for all $k,z$
    , and $\displaystyle\sup_{|z|<\delta}|\psi_k(z)-\psi(z)|\underset{k\to\infty}{\longrightarrow}0$, then
    $$\int_{0<|z|<\delta}\psi_k(z)\,\d\mu_{1,k}(z)\underset{k\to\infty}{\longrightarrow}
    \int_{0<|z|<\delta}\psi(z)\,\mu_{1}(\!\dz)\,.$$ \\[2mm]
    $(d)$ If $\|\psi_k\|_\infty\leq C$ for all $k$, and
    $\displaystyle\sup_{|z|<K}|\psi_k(z)-\psi(z)|\underset{k\to\infty}{\longrightarrow}0$ for any $K>0$,
    then 
    $$\int_{|z|>0}\psi_k(z)\,\d\mu_{2,k}(z)\underset{k\to\infty}{\longrightarrow}
    \int_{|z|>0}\psi(z)\,\mu_{2}(\!\dz)\,.$$
\end{lemma}

We also need the following results.

\begin{lemma}\label{app:lem.reg}
   Assume \hyp{M}, \hyp{J1'},\hyp{J2'} and let $v\in
   C^{1,\theta}(\overline{B}_R)$ for some $\theta\in(0,1)$. Then\\[2mm]
    $(a)$ the function $x\mapsto L_{k}^R[v,Dv](x)$ belongs to
    $C^{0,\theta}(\overline{B}_R)$;\\[2mm]
    $(b)$ if $v_n\to v$ in $C^{1,\theta}(\overline{B}_R)$, 
    then $L_{k}^R[v_n,Dv_n]\to L_{k}^R[v,Dv]$ in $C^{0,\theta}(\overline{B}_R)$.
\end{lemma}
\begin{proof}
 $(a)$\quad   We only do the proof for the $L_1$-term since the  $L_2$-case
 is similar but easier. Below $x,y\in \ol B_R$, $\frac1k<|z|<k$ (the
 support of $\mu_{1,k}$), and $C_*$ will denote all constants (that may vary from
line to line) depending only on $R,k,\eps,v,j_i,\mu_i,N$.  Since $v$
is $C^{1,\theta}$, 
% $|Dv(x)|,|Dv(y)|\leq C_*$,
% $$|v(x)-v(y)|\leq
%     C_*|x-y| \qquad \text{and}\qquad|Dv(x)-Dv(y)|\leq
%     C_*|x-y|^\theta.$$
$$|v(x)|+|Dv(x)|+\frac{|Dv(x)-Dv(y)|}{|x-y|^{\theta}}\leq\|v\|_{C^{1,\theta}}.$$
Then, $\big|v(x)-v(y)\big| \leq C_*\big|x-y\big|$
and by Lipschitz continuity of $P_R$ and assumption \hyp{J2'}, 
\begin{align*}
\Big|v\Big(P_R\big(x+j_{1}(Dv(x),z)\big)\Big)-v\Big(P_R\big(y+j_{1}(Dv(y),z)\big)\Big)\Big|&\leq
    C_*\big(|x-y|+ |Dv(x)-Dv(y)|\big)\,,\\
    \big|Dv(x)\cdot j_{1}(Dv(x),z)-Dv(y)\cdot j_{1}(Dv(y),z)\big|&\leq
    C_*\big|Dv(x)-Dv(y)\big|\,.
\end{align*}
Thus, all these quantities are controlled by $C_*|x-y|^\theta$.
Since the measure $\mu_{1,k}$ is bounded and supported in
$\frac1k<|z|<k$, it then follows that
    $$\big|L^R_{1,k}[v,Dv](x)-L^R_{1,k}[v,Dv](y)\big|\leq C_*|x-y|^\theta\mu_{1,k}(\R^N), $$
and the proof of $(a)$ is complete.
\smallskip

\noindent $(b)$\quad 
By assumption
$$|v_n(x)-v(y)|+\frac{|v_n(x)-v(y)|}{|x-y|}+
\frac{|Dv_n(x)-Dv(y)|}{|x-y|^{\theta}}\leq\|v_n-v\|_{C^{1,\theta}}\to
 0\qquad\text{as}\qquad n\to\infty.$$ 
By similar computations as above, we end up with
    $$\big|L_{k}^R[v_n,Dv_n](x)-L_{k}^R[v,Dv](y)\big|\leq
    C_*\|v_n-v\|_{C^{1,\theta}}\mu_{1,k}(\R^N)|x-y|^\theta\,,$$
    for a $C_*$ independent of $x,y,n$. It follows that
    $L_{k}^R[v_n,Dv_n]\to L_{k}^R[v,Dv]$ in 
    $C^{0,\theta}(\overline{B}_R)$ as $n\to\infty$.
\end{proof}

We can now prove an existence result for the approximate problem
\eqref{approx.pb}. 

\begin{proposition}\label{lem:existence.approx.pb}
    Assume \hyp{M}, \hyp{J1'}, \hyp{J2'}, \hyp{F5'}, and let $\eps,R,M>0$,
    $k\in\mathbb{N}$. Then there
    exists a classical solution $u\in C^{2}(\overline B_R)$ of
    \eqref{approx.pb}. 
\end{proposition}

\begin{proof} The proof is  based on Schauder's fixed point
    theorem (cf. e.g. \cite[Corollary 11.2]{GT:Book}).

    \noindent \textbf{1)}\quad Let $X:=C^{1,\theta_0}(\overline B_R)$ for
    a fixed $\theta_0\in(0,1)$, and define
    $$\mathcal{C}=\big\{ w\in X: \|w\|_\infty\leq M+\|f\|_\infty\big\}\,.$$
    Note that $\mathcal{C}$ is a convex and closed subset of
    $X$. On $\mathcal{C}$ we now define a map $\cT=\cT_{\eps,R,M,k}$ in
    the following way: for every $v\in \mathcal C$,  $u=\cT(v)$ is the
    classical solution of the Dirichlet problem  
    \begin{align}
\label{iteration}\begin{cases}
        u-\eps\Delta u = T_M\Big[L_{k}^R[v,Dv]\Big]+f(x)& \text{ in } B_R \,,\\
        u = 0 & \text{ on }\partial B_R\,.
    \end{cases} 
    \end{align}
    When $v\in X$, $L_{k}^R[v,Dv]\in C^{0,\theta_0}$ by Lemma
    \ref{app:lem.reg}, and then by the definition of $T_M$ and \hyp{F5'},
$$w:= T_M\Big[L_{k}^R[v,Dv]\Big]+f(x)\in C^{0,\theta_0}(\ol B_R).$$
Since $B_R$ is a smooth domain, classical results (\cite[Corollary
6.9]{GT:Book}) then tell us that there exists a unique classical solution
$u\in C^{2,\theta_0}(\overline B_R)\ (\subset X)$ of
\eqref{iteration}. Moreover, by the maximum principle and the
definition of $T_M$, $\|u\|_\infty\leq M+\|f\|_\infty$. We conclude
that $\cT$ is a well-defined map from $\mathcal{C}$ into
$\mathcal{C}$. 
\smallskip

    \noindent \textbf{2)}\quad  We show that
    $\cT:\mathcal{C}\to\mathcal{C}$ is continuous with respect to norm
    of $X=C^{1,\theta_0}(\ol B_R)$. Take a sequence
    $\{v_n\}\subset\mathcal{C}$ such that $v_n\to v$ in $X$. By subtracting the equations for
    $u_n$ and $u_p$, we see that $w:=u_n-u_p$ is a classical solution of
    $$\begin{cases}
w-\eps\Delta w =
    T_M\big[L_{k}^R[v_n,Dv_n]\big]-T_M\big[L_{k}^R[v_p,Dv_p]\big]=:g_{n,p},\quad&
    \text{in } B_R,\\
w=0, &\text{in } \partial B_R. 
\end{cases}
    $$
By the maximum principle, we then find that
    $$\|w\|_{C^0(\overline{B}_R)}\leq
    \|g_{n,p}\|_{C^0(\overline{B}_R)},$$
and by standard $C^{2,\theta_0}$-theory (e.g. \cite[Thm
6.6]{GT:Book}),
    $$\|u_n-u_p\|_{C^{2,\theta_0}(\overline B_R)}=\|w\|_{C^{2,\theta_0}(\overline B_R)} 
    \leq C\Big(\big\|w\big\|_{C^0(\overline{B}_R)}+
    \|g_{n,p}\|_{C^{0,\theta_0}(\overline B_R)}\Big)\,.$$
Hence, since $T_M\big[L_{k}^R[v_n,Dv_n]\big]\to
    T_M\big[L_{k}^R[v,Dv]\big]$ in $C^{0,\theta_0}(\overline B_R)$ by Lemma
    \ref{app:lem.reg} $(b)$, it follows that $(u_n)$ is a Cauchy
    sequence in $C^{2,\theta_0}(\overline B_R)$ and hence also in
    $X$. By completeness, the limit $u$ exists and belongs to
    $C^{2,\theta_0}(\overline B_R)\cap \mathcal C$ since $\mathcal
    C$ is closed in $X$. 
 
By the $C^{2,\theta_0}$ convergence of $u_n$, the $C^{1,\theta_0}$
convergence of $v_n$, and Lemma \ref{app:lem.reg}, we can pass to the
limit in the equation to see that $u=\cT(v)$. It follows that
$\cT(v_n)\to \cT(v)$ in $X$, and we conclude that $\cT$
is continuous in $\mathcal C$. 
\smallskip

    \noindent \textbf{3)}\quad We now show that $\cT(\mathcal{C})$ is
    relatively compact. Take any sequence $\{u_n\}\subset
    \cT(\mathcal{C})$. Then there exists a sequence
    $\{v_n\}\subset\mathcal{C}$ such that $u_n=\cT(v_n)$. It follows that
    $$u_n-\eps\Delta u_n = g_n\,,$$
    where $|g_n|\leq M+\|f\|_\infty$. By the maximum principle and $W^{2,p}$-theory
    (e.g. \cite[Thm 9.11]{GT:Book}), we have the following two a priori
    estimates for any $1<p<\infty$,
\begin{align*}
&\|u_n\|_{L^\infty(B_R)}\leq M+\|f\|_\infty,\\
&\|u_n\|_{W^{2,p}(B_R)}\leq
    C\Big(\|u_n\|_{L^p(B_R)}+\|g_n\|_{L^p(B_R)}\Big)
    \leq 2C|B_R|^{1/p}\big(M+\|f\|_\infty\big).
\end{align*}
    By compact embeddings of Sobolev spaces \cite[Thm 7.26]{GT:Book},
    we can extract a subsequence $u_{\varphi(n)}$ converging in 
    $C^{1,\theta}(\overline B_R)$ for any $\theta\in(0,1)$, in
    particular for $\theta=\theta_0$. This provides 
    a subsequence which converges in $X$, and proves the claim.

    \noindent \textbf{4)}\quad By Schauder's fixed point theorem, there exists a function
    $u\in\mathcal{C}$ such that $u=\cT(u)$, which means that we have a
    $C^2(\overline B_R)$-solution of \eqref{approx.pb} and the proof
    is complete. 
\end{proof}

We proceed to prove existence under the restrictive assumptions
\hyp{J1'} and \hyp{J2'}. We will need the following result.

\begin{lemma}\label{lem:comp.approx.pb}
Assume \hyp{M}, \hyp{J1'}, $f,g\in
    C^0(\overline{B}_R)$. Let $u$ and $v$ be $C^2(\overline B_R)$
    solutions of
\begin{align*}
u-T_M\Big[L^R_{k}[u,Du]\Big]-\eps\Delta u\leq f(x) \qquad\text{and}\qquad
  v-T_M\Big[L^R_{k}[v,Dv]\Big]-\eps\Delta v\geq g(x)\qquad \text{in}\ \ B_R.
\end{align*}
If $u\leq v$ on $\partial B_R$ and $f\leq g$ in $B_R$, then $u\leq v$
in $\ol{B}_R$.
\end{lemma}
\begin{remark}
The result is a comparison result for smooth sub and supersolutions of
\eqref{approx.pb}. It implies uniqueness of classical solutions
of \eqref{approx.pb}. 
\end{remark}
\begin{proof}
Let $w:=u-v$ and $\bar m:=\max_{\overline{B}_R}w$. If this max is attained at
$x_0\in\partial B_R$, since $u\leq v$ there, we have $\bar m\leq0$. 
Otherwise, there is an interior point $x_0\in B_R$ such that $m=w(x_0)$. By assumption,
    $$w(x_0)-\Big\{T_M\big[L_{k}^R[u,Du]\big](x_0)- 
        T_M\big[L_{k}^R[v,Dv]\big](x_0)\Big\}-\eps\Delta w(x_0)\leq
        (f-g)(x_0).$$
Since $x_0$ is a maximum point, and $u,v$ are smooth, 
$Dw(x_0)=(Du-Dv)(x_0)=0$, $\Delta w(x_0)\geq 0$,
and $(u-v)(x_0)\geq (u-v)(y)$ for all $y\in \ol B_R$. By the latter
inequality,
$$u\Big(P_R\big(x_0+j(Du(x_0),z)\big)\Big)-u(x_0)\leq
v\Big(P_R\big(x_0+j(Dv(x_0),z)\big)\Big)-v(x_0),$$
and hence
$$L_{k}^R[u,Du](x_0)-L_{k}^R[v,Dv](x_0)\leq 0.$$
Since $T_M$ is a non-decreasing function and $f\leq g$, we can conclude that
$$\bar m=w(x_0)\leq 0, $$
and the proof is complete because in either case, we get $\bar m\leq0$.
\end{proof}

\begin{corollary}\label{cor:bounds}
    If $u_{R,M,k,\eps}$ is the solution of \eqref{approx.pb}, then
    $\|u_{R,M,k,\eps}\|_{L^\infty(B_R)}\leq\|f\|_{L^\infty(\R^N)}\,.$
\end{corollary}
\begin{proof}
Follows from Lemma \ref{lem:comp.approx.pb} with $-\|f\|_\infty$/$u$/$\|f\|_\infty$ as subsolution/solution/supersolution.
\end{proof}

\begin{proposition}\label{lem:limit.existence}
    Assume \hyp{M}, \hyp{J1'}, \hyp{J2'}, \hyp{J2}, and
    \hyp{F5'}. Then there exists a unique viscosity
    solution $u\in C_b(\R^N)$ of \eqref{eq:main.nonlocal}. 
\end{proposition}

\begin{proof}
    Let $R>0$, $k=M=1/\eps=n\in\N$, and $u_{R,n}$ be the corresponding
    solution of \eqref{approx.pb} given by Proposition
    \ref{lem:existence.approx.pb}. Using the ``half relaxed limit''
    method, we first we send $R\to\infty$ and then
    $n\to\infty$ and show that we can obtain from $u_{R,n}$ a function $u$
    which is the viscosity solution of \eqref{eq:main.nonlocal}.

\noindent {\bf 1)} {\em Claim:} For every $n\in\N$, the functions
    $$\overline{u}_n(x):=\limsup_{y\to x,R\to\infty}
    u_{R,n}(y)\,\qquad\text{and}\qquad
    \underline{u}_n(x):=\liminf_{y\to x,R\to\infty} u_{R,n}(y)\,$$
are bounded viscosity sub- and supersolutions respectively of
\begin{align}
\label{intermediate}
u-\tfrac1n\Delta
  u-T_n\Big[L_n[u,Du]\Big]=f\qquad\text{in}\qquad\R^N.
\end{align}
where $L_{k}$ is defined as $L$ in \eqref{I-def}, but with the measure
$\mu_{i,k}$ replacing $\mu_i$ for $i=1,2$. Moreover,
$$-\|f\|_\infty\leq \ul u_n(x)\leq\ol u_n(x)\leq \|f\|_\infty\qquad\text{in}\qquad\R^N.$$

\noindent {\em Proof of Claim:} First note that $\ol u_n$ and $\ul
u_n$ are defined for every $x\in\R^N$ since $R\to\infty$, they are
semicontinuous and $\ul
u_n\leq\ol u_n$ by definition and bounded by $\|f\|_\infty$ by
Corollary \ref{cor:bounds}.  We show that $\ol u$ is a subsolution of
\eqref{intermediate}  according to Definition
\ref{ineq:subsol.nolocal_2}. Take any bounded test function
$\phi$ and any point $x\in\R^N$ such that $\overline{u}-\phi$ has a
global maximum at $x$. We may as usual assume the maximum is
strict. Then there exists a sequence $y_R\to x$ 
    of maximum points of $u_{R,n}-\phi$ in $\overline{B}_R$, such that
    $u_{R,n}(y_R)\to\overline{u}(x)$. Take $R$ big enough such that
    $R>|x|$ and $y_R\in B_R$ (since $y_R\to x$, $y_R$ cannot be located on
    $\partial B_R$ for $R$ big enough).
    Since $y_R$ is a maximum point and $u_{R,n}$ is smooth, $Du_{R,n}(y_R)=D\phi(y_R)$, $\Delta
u_{R,n}(y_R)\leq \Delta \phi(y_R)$, and (cf. the proof of Lemma
\ref{lem:comp.approx.pb}) 
    $$L_n^R[u_{R,n},Du_{R,n}](y_R)\leq L_n^R[\phi,D\phi](y_R)\,.$$
Since $u_{R,n}$ satisfies equation \eqref{approx.pb} at the point $y_R$, it then
follows that
\begin{align}
\label{int_ineq}
u_{R,n}(y_R)-\tfrac1n\Delta\phi(y_R)-T_n\Big[L_n^R[\phi,D\phi](y_R)\Big]\leq
    f(y_R)\,.
\end{align}
By the boundedness of $y_R$, the regularity of $\phi$, \hyp{J1'}, and
the definition of $P_R$,  
$$\phi\Big(P_R\big(y_R+j_{i}(D\phi(y_R),z)\big)\Big) \to
\phi\big(x+j_{i}(D\phi(x),z)\big) \quad\text{as}\quad
R\to\infty\,,\quad\text{for}\quad a.e.\  z \text{ and } i=1,2.$$
Hence, since this term is uniformly bounded, $\mu_{i,k}$ is bounded, and
\hyp{M} holds, we can use the dominated convergence theorem to
conclude that
$$L_n^R[\phi,D\phi](y_R)\to
L_n[\phi,D\phi](x)\qquad\text{as}\qquad R\to\infty. $$ 
By the regularity of $\phi$ and the continuity of $T_n$, we can then
pass to the limit as $R\to\infty$ in \eqref{int_ineq} and find that
    $$\ol u(x)-\tfrac1n\Delta\phi(x)-T_n\Big[L_n[\phi,D\phi](x)\Big]\leq
    f(x)\,.$$
We conclude that $\ol u$ is a viscosity subsolution of
\eqref{intermediate}, and in a similar way we can show that $\ul u$ is viscosity
supersolution of \eqref{intermediate}. The claim is proved.
\smallskip

\noindent {\bf 2)}\quad We now pass to the limit as $n\to\infty$. We
proceed as before, defining
    $$\overline{u}(x):=\limsup_{y\to x,n\to\infty}
    \ol u_{n}(y)\,\qquad\text{and}\qquad
    \underline{u}(x):=\liminf_{y\to x,n\to\infty} \ul u_{n}(y)\,,$$
where $\ol u_n$ and $\ul u_n$ are the uniformly bounded sub and supersolutions
of \eqref{intermediate} given by part 1). It immediately follows that
$-\|f\|_\infty\leq \ul u\leq \ol u\leq \|f\|_\infty$.

We prove that $\ol u$ is
viscosity subsolution of \eqref{eq:main.nonlocal}. Again we take any
smooth bounded test function $\phi$ and global strict maximum point $x$ of
$\overline{u}-\phi$ in $\R^N$. We have a sequence $y_n\to x$ of
maximum points of $u_n-\phi$ such that $u_n(y_n)\to\overline{u}(x)$.
    Fix $\delta>0$ and choose $n$ big enough so that
    $1/n<\delta$. We split $L_{1,n}$ into
    $L_{1,n}=L_{\delta,n}+L^\delta_{n}$, where $L_{\delta}$ and
    $L^\delta$ are defined in the beginning of section
    \ref{sect:visco}. Since $\ol u$ is a subsolution of
    \eqref{intermediate}, it follows that
\begin{align}
\label{int_ineq2} u_n(y_n)-\tfrac1n\Delta\phi(y_n)-T_n\Big[L_n[\phi,D\phi](y_n)\Big]\leq
    f(y_n)\,.\end{align}

We pass to the limit in the different terms of this inequality. First,
let 
$$\psi_n(z)=\phi\big(y_n+j_{1}(D\phi(y_n),z)\big) - \phi(y_n)- D\phi(y_n)\cdot 
    j_{1}\big(D\phi(y_n),z\big)\ind{|z|<1}.$$
Since $D^2\phi$ is continuous, $y_n\to x$ and $D\phi(y_n)\to D\phi(x)$
are bounded, a Taylor expansion and \hyp{J1'} reveals that there is a
$C$ independent of $n$ and $z$ such that
$$|\psi_n(z)|\leq
\sup_{t\in(0,1)}\tfrac12\Big|D^2\phi\Big(y_n+tj_{1}\big(D\phi(y_n), z\big)\Big)\Big|\big|j_{1}\big(D\phi(y_n),z\big)\big|^2\leq
C|z|^2\quad\text{for}\quad |z|<1.$$
    By regularity of $\phi$ and \hyp{J2'}, the $\psi_n$ converge
    uniformly for $|z|\leq \delta$ to  
    $$
        \psi(z):=\phi\big(x+j_{1}(D\phi(x),z)\big) - \phi(x)-
    D\phi(x)\cdot j_{1}\big(D\phi(x),z\big)\ind{|z|<1}\,,$$
and hence we use Lemma \ref{lem:measure} $(c)$ to conclude that
 $$L_{\delta,n}[\phi,D\phi](y_n)\to L_\delta[\phi,D\phi](x).$$
Simpler but similar arguments, this time using Lemma \ref{lem:measure}
$(b)$ and $(d)$, show that 
$$L^\delta_{n}[\phi,D\phi](y_n)\to L^\delta[\phi,D\phi](x) \quad\text{and}\quad
L_{2,n}[\phi,D\phi](y_n)\to L_2[\phi,D\phi](x).$$
We conclude that $L_n[\phi,D\phi](y_n)\to L[\phi,D\phi](x)$, and since
this limit is bounded, for $n$ big enough
$$T_{n}[L_n[\phi,D\phi](y_n)]=L_n[\phi,D\phi](y_n)\to
L[\phi,D\phi](x)\qquad\text{as}\qquad n\to\infty.$$ 
In view of the regularity of $\phi$ and $f$, we can then pass to the
limit as $n\to\infty$ in \eqref{int_ineq2} to find that
  $$\overline{u}(x)-L[\phi,D\phi](x)\leq f(x)\,.$$
Hence $\overline{u}$ is a viscosity subsolution of \eqref{eq:main.nonlocal}.
    Similar arguments show that $\underline{u}$ is viscosity supersolution.
\smallskip

\noindent {\bf 3)}\quad   By the comparison result for
    semicontinuous viscosity solutions, Theorem~\ref{thm:comp} $(a)$,
    it follows that $\overline{u}\leq \underline{u}$. Note that this
    result requires \hyp{J2}, see also Remark \ref{rem:newass}. Since the
    opposite inequality holds by part 2),
    $\overline{u}=\underline{u}=:u$, and this function is continuous,
    bounded by $\|f\|_\infty$, and a viscosity solution of
    \eqref{eq:main.nonlocal}. The proof is complete.
\end{proof}

Now we show how to remove assumptions \hyp{J1'}, \hyp{J2'}, and \hyp{F5'} and
obtain an existence result in the general case. We do it in two steps,
starting by removing \hyp{J2'} and \hyp{F5'} by a regularization
argument (mollification), and then we remove \hyp{J1'} by a truncation argument.
We will need the following lemma, whose proof is given in Appendix
\ref{sec:pf2}. 

\begin{lemma}\label{lem:reg.j} Assume \hyp{M}, \hyp{J1'}, and \hyp{J2}
  and define $j_{1,k}$ and $j_{2,k}$ by
$$j_{i,k}(p,z):=\big(\rho_k\ast
j_{i}(\cdot,z)\big)(p)=\int_{\R^N}\rho_k(q-p)j_i(q,z)\,\d
    q\qquad\text{for}\qquad i=1,2,$$
for a mollifier $\rho_k(p)=k^N\rho(kp)$, $0\leq \rho\in
C^{\infty}(\R^N)$ is symmetric with support in $B_1$ and $\int\rho
=1$. Then the following holds: 

\noindent $(a)$ For any $k\in\N$, $j_{1,k}$ and $j_{2,k}$ are locally bounded,
and for any $|p|\leq r$, $|z|<1$, and $k\in\N$, 
$$|j_{1,k}(p,z)|\leq C_{r+1/k}|z|,\quad\text{where}\quad C_r \quad\text{is
    the constant from \hyp{J1'}}.$$

\noindent $(b)$ For every $k\in\N$ and $r>0$, there is $C_{k,r}$ such that for all
$|p|,|q|,|z|\leq r$,
$$|j_{i,k}(p,z)-j_{i,k}(q,z)|\leq
C_{k,r}|p-q|\qquad\text{for}\qquad i=1,2. $$

\noindent $(c)$ For any $r>0$, $|p|,|q|<r$ and $k\in\N$, 
    $$\int_{|z|>0}|j_{1,k}(p,z)-j_{1,k}(q,z)|^2\dmu_1(z)\leq 
    \omega_{j,r+\frac1k}(p-q)\quad\text{where}\quad \omega_{j,r} \quad\text{is
    the modulus from \hyp{J2}}\,.$$

\noindent $(d)$ There exists $\delta_0>0$ such that for any
    $r>0$ and $\eps>0$ there exists $\eta>0$ such that
    $$\sup_{|p|<r,k\in\N}\int_A|j_{1,k}(p,z)|^2\,\d\mu_1(z)<\eps\,$$
 for every Borel set $A\subset \{0<|z|<\delta_0\}$ such that
 $\int_A|z|^2\dmu_1(z)<\eta$.  

\noindent $(e)$ Let $L_{1,k}$ and $L_{2,k}$ be defined as in \eqref{I-def} and \eqref{I-def2}
with $j_1$ and $j_2$ replaced by $j_{1,k}$ and $j_{2,k}$.
For all $\phi\in C_b\cap C^2$ and $x_k\to x$,
$$L_{i,k}[\phi,D\phi(x_k)](x_k)\to L_i[\phi,D\phi(x)](x)\qquad \text{for}\qquad i=1,2.$$

\end{lemma}

Here is the general existence result for equation \eqref{eq:main.nonlocal}.

\begin{proposition}\label{prop:ex}
    Assume \hyp{M}, \hyp{J1}--\hyp{J3} and \hyp{F1}--\hyp{F5}. 
    Then there exists a viscosity solution of \eqref{eq:main.nonlocal}.
\end{proposition}

\begin{proof} We proceed in two steps.
    
    \noindent\textbf{1)} Assume in addition \hyp{J1'} holds. We
    approximate $j_i$ and $f$ by 
 $$j_{i,k} \quad\text{as defined in Lemma
   \ref{lem:reg.j}}\qquad\text{and}\qquad f_k=\rho_k\ast f,$$ 
where $k\in\N$ and $\rho_k$ is the mollifier defined in Lemma
\ref{lem:reg.j}. By Lemma \ref{lem:reg.j}, $j_{i,k}$
    satisfy \hyp{J1}--\hyp{J3} and \hyp{J1'}--\hyp{J2'}, while assumption
 \hyp{F5} and properties of mollifiers imply that $f_k$ is bounded and
 Lipschitz continuous in $\R^N$ (so \hyp{F5'} holds) and converges uniformly to $f$ in $\R^N$ as $k\to\infty$.

Let $L_{k}$ be defined as $L$ with $j_{i,k}$ replacing $j_i$, and
consider the problem
    $$u-L_{k}[u,Du]=f_{k}\qquad\text{in}\qquad \R^N.$$
By the above discussion, this problem satisfies all the assumptions of 
Proposition \ref{lem:limit.existence} for every $k\in\N$, and hence
 there exist solutions $u_k$ of this problem for every $k\in\N$. By
 the comparison principle (Theorem \ref{thm:comp} $(a)$), 
$$\|u_k\|_\infty\leq\|f_{k}\|_\infty\leq\|f\|_\infty.$$
As in the proof of Proposition~\ref{lem:limit.existence}, we then define the
half-relaxed limits
$$\overline{u}(x)=\limsup_{y\to
  x,k\to\infty}u_k(y)\qquad\text{and}\qquad \underline{u}(x)=\liminf_{y\to x,k\to\infty}u_k(y),$$
prove that they are viscosity sub and supersolutions of
\eqref{eq:main.nonlocal}, and conclude by the comparison result
Theorem \ref{thm:comp} $(a)$ that $\ol u=\ul u=u$ is a viscosity solution of
\eqref{eq:main.nonlocal}. We just sketch the argument for $\ol u$
being a subsolution. Take a smooth test-function $\phi$ such that
$\overline{u}-\phi$ has a strict maximum at $x\in\R^N$. Then there
exists a sequence $y_k\to x$ such 
    that $u_k -\phi$ has a maximum at $x_k$. Using the subsolution property for
    $u_k$, the fact that $L_{k}[\phi,D\phi](x_k)\to L[\phi,D\phi](x)$
    (Lemma~\ref{lem:reg.j} $(e)$) and that $f_{k}(x_k)\to f(x)$, we
    arrive at 
    $$\overline{u}(x)-L[\phi,D\phi](x)\leq f(x)\,,$$
    hence $\overline{u}$ is a
    subsolution of \eqref{eq:main.nonlocal}. This completes the
    existence proof under the assumptions \hyp{M}, \hyp{J1'}, \hyp{J2},
    \hyp{J3}, and \hyp{F5}.
\smallskip

\noindent\textbf{2)} 
We remove the \hyp{J1'} condition through a truncation procedure for
the $j_1$-term: 
$$j_{1}^M(p,z):=|z|\,T_M\Big(\frac{j_1(p,z)}{|z|}\Big)\,,$$
where $T_M(x)=x$ if $|x|\leq M$ and $T_M(x)=M\frac x{|x|}$ if
$|x|>M$. In this case is easy to see that $j_{1}^M$ satisfies
\hyp{J1}--\hyp{J3} and in addition \hyp{J1'} with $C_r=M$.
Furthermore, for almost any $z$, 
$$j_{1}^M(p,z)\to j_1(p,z)\quad\text{locally uniformly in}\quad p.$$ 
Hence by part 1), for any $M>0$ there exists a solution $u_M$ of 
$$u-L^M[u,Du]=f\qquad\text{in}\qquad\R^N,$$
where $L^M$ is defined as $L$ (cf. \eqref{I-def}) but with
$(j_1^M,j_2)$ replacing $(j_1,j_2)$. As in part 1), we define the
half-relaxed limits 
$$\overline{u}(x)=\limsup_{y\to
  x,M\to\infty}u_M(y)\qquad\text{and}\qquad \underline{u}(x)=\liminf_{y\to x,M\to\infty}u_M(y),$$
prove that they are viscosity sub and supersolutions of
\eqref{eq:main.nonlocal}, and conclude by the comparison result
Theorem \ref{thm:comp} $(a)$ that $\ol u=\ul u=u$ is a viscosity solution of
\eqref{eq:main.nonlocal}. In view of previous arguments, the only thing we need to check is that
$$L_1^M[\phi,D\phi](x_M)\to L_1[\phi,D\phi](x)\quad\text{when}\quad
x_M\to x.$$

To do so, we use note that the integrand converges pointwise a.e. by
the regularity of $\phi$ and the local uniform convergence of $j_1^M$
in the $p$-variable. Since $|j_1^M(p,z)|\leq|j_1(p,z)|$, $j_1^M$
satisfies the equi-integrability condition \hyp{J3} uniformly in $M$. Hence we may use Vitali's convergence theorem
(cf. Appendix \ref{app:equiint} and Remark \ref{rem:assump}-$(b)$) to
pass the limit 
$M\to\infty$ inside the integral over $|z|<\delta$. Passage to the
limit in the integral over $|z|\geq\delta$ is done using the
dominated convergence theorem since $\phi$ is bounded and $\mu_1$ is not
singular on this domain. 
The proof is complete.
\end{proof}

\begin{proof}[Proof of Theorem~\ref{thm:existence}] Uniqueness and
  uniform continuity follow from Corollary \ref{cor:main}. Existence of
  solutions of \eqref{eq:main.nonlocal} follows from
  Proposition~\ref{prop:ex}. Existence for \eqref{eq:0} follows in a
  similar way as for \eqref{eq:main.nonlocal} where all difficulties
  are already present. We only give a very brief sketch 
of the proof. We start by an approximate problem a la
\eqref{approx.pb}. In the general case it is to find $u\in C^2(\overline B_R)$ 
such that
\begin{equation*}\label{approx.pb2}
 \begin{cases}  \gamma u+T_M\Big[-\gamma u+
     F\big(u,Du,L^R_{k}[u,Du]\big)\Big]-\eps\Delta u=f(x)\,,&\quad
        x\in B_R(0)\,,\\
u=0,&  \quad       x\in \partial B_R(0)\,,
\end{cases}
\end{equation*}
where $L^R_{k}$ is defined below \eqref{approx.pb}.
Under the same assumptions as the linear case in addition to the
assumption that $F$ is also locally Lipschitz, we obtain existence of
solutions of this problem following step by step the fixed point
argument of the proof of Proposition \ref{lem:existence.approx.pb}. To
get an existence 
result for equation \eqref{eq:0}, we follow the approximation and limit
procedures given in the proofs of Proposition \ref{lem:limit.existence} and
\ref{prop:ex}. The
only slight difference is that we also need an approximation argument for $F$
(e.g. by mollification again), and when we pass to the limit, we use this time
the strong comparison result Theorem~\ref{thm:comp}-$(b)$ for the limit
equation. It is straight forward to check that this will work out and
produce a bounded viscosity solution of \eqref{eq:0}. 
\end{proof}

%----------------------------------------------------------------------
\subsection{Proof of  Theorem \ref{thm:localization} (local limits)}
\label{sect:localization}

Theorem \ref{thm:localization} is a consequence of the half-relaxed limit method and
comparison. Define 
\begin{align*}
\bar{u}(x):=\limsup_{\eps\to0, y\to x}u_\eps(y)\qquad\text{and}\qquad
        \underline{u}(x):=\liminf_{\eps\to0, y\to x}u_\eps(y)\,.
\end{align*}

In the quasilinear case we have the following result.
\begin{lemma}\label{lem:limit} Assume the assumptions of Theorem
    \ref{thm:localization}-$(a)$ hold.
	Then $\bar{u}$ is a viscosity subsolution of \eqref{eq:local} and 
        $\underline{u}$ is a viscosity supersolution of	\eqref{eq:local}.
\end{lemma}

\begin{proof}
 Let $\phi$ be a smooth bounded function such that $\bar u-\phi$ has a
 global maximum at $x$. We may assume the maximum is strict and that there
        is a sequence $\{y_\eps\}_\eps$ of global
        maximum points of $u_\eps-\phi$ such that $y_\eps\to x$ and
        $u_\eps(y_\eps)\to \bar u(x)$ as $\eps\to0$. 
        We let $L_{1,\eps}=L_{\eps,\delta}+L_{\eps}^{\delta}$ as in
        \eqref{def:L.delta}, and to see the localization effect, we
        also decompose  
        $L_{2,\eps}=\tilde L_{\eps,\delta}+\tilde 
        L_{\eps}^{\delta}$. By the maximum point property $\tilde
        L_{\eps,\delta}[u_\eps,D\phi](y_\eps)\leq 
        \tilde L_{\eps,\delta}[\phi,D\phi](y_\eps)$, and then since
        $u_\eps$ is a subsolution of \eqref{eq:lineps},
 	\begin{align}
	\label{v_ineq}
        -L_{\eps,\delta}[\phi,D\phi](y_\eps)-L_{\eps}^{\delta}[u_\eps,D\phi](y_\eps)-\tilde
        L_{\eps,\delta}[\phi,D\phi](y_\eps)-\tilde
        L_{\eps}^\delta[u_\eps,Du_\eps](y_\eps)+ u_\eps(y_\eps) \leq
        f(y_\eps)\,.
	\end{align}
Then by a Taylor expansion and \hyp{J4}, for $i=1,2$,
\begin{align*}
j_i(D\phi(y_\eps),z)=\sigma_i(D\phi(y_\eps))z+o(z)\qquad\text{as}\qquad
z\to 0,
\end{align*}
where $o(z)$ is independent of $\eps>0$. Another Taylor expansion and
hypotheses \hyp{\Meps} applied to 
$Y=\sigma_1\big(D\phi(x)\big)^TD^2\phi(x)\,\sigma_1\big(D\phi(x)\big)\in\M{P}{P}$
then gives that
	\begin{align*}
	L_{\eps,\delta}[\phi,D\phi](y_\eps)
        &=\frac{1}{2}\int_{|z|<\delta}z^T\Big[\sigma_1\big(D\phi(y_\eps)\big)^T
         D^2\phi(y_\eps)\,\sigma_1\big(D\phi(y_\eps)\big)\Big]z\,\dmu_{1,\eps}(z)+o_\delta(1)\\  
         &=\frac{1}{2}\tr \Big[
         A_1^T\Big[\sigma_1\big(D\phi(x)\big)^T
         D^2\phi(x)\,\sigma_1\big(D\phi(x)\big)\Big]A_1 \Big]+o_\eps(1) +o_\delta(1)\\
         &=\frac{1}{2}\tr \Big[
         \tilde\sigma_1\big(D\phi(x)\big)
         \sigma_1\big(D\phi(x)\big)^TD^2\phi(x) \Big]+o_\eps(1) +o_\delta(1)\,,
	 \end{align*}
         where for the last line, we use that $\tilde\sigma_1(x,p)=\sigma_1(x,p)A_1$
         is a $N\times N$ matrix as is $D^2\phi(x)$, so that we can use the
         property $\tr(M_1M_2)=\tr(M_2M_1)$.  Here the $o_\delta(1)$-term is
         independent of $\eps$
        since the measure  $|z|^2\mu_{1,\eps}(\dz)$ has a uniformly
        bounded mass. Similarly, by \hyp{J4} and \hyp{\Meps},
	\begin{align*}
	\tilde L_{\eps,\delta}[\phi,D\phi](x)
        &=\frac{1}{2}\int_{|z|<\delta} z^T\Big[\sigma_2\big(D\phi(x)\big)^T 
        D^2\phi(x)\,\sigma_2\big(D\phi(x)\big)\Big]z\,\dmu_{2,\eps}(z)\\    
        &\quad+\int_{|z|<\delta} D\phi(x)
        \sigma_2\big(D\phi(x)\big)z\,\dmu_{2,\eps}(z)+o_\eps(1)  +o_\delta(1)\\
        &=\frac{1}{2}\tr \Big[
        A_2^T\Big[\sigma_2\big(D\phi(x)\big)^TD^2\phi(x)\,
        \sigma_2\big(D\phi(x)\big)\Big]A_2\Big]\\
        &\quad +\sigma_2\big(D\phi(x)\big)D\phi(x)\cdot a+o_\eps(1)
        +o_\delta(1)\\[2mm]
        &=\frac{1}{2}\tr \Big[
        \tilde\sigma_2\big(D\phi(x)\big) \tilde\sigma_2\big(D\phi(x)\big)
        D^2\phi(x) \Big]+b\big(D\phi(x)\big)\cdot D\phi(x)+o_\eps(1) +o_\delta(1)\\
	\end{align*}
By \hyp{J1} and continuity, $D\phi(y_\eps)$ and
$j(D\phi(y_\eps),z)$ are uniformly bounded in $\eps>0$
for $|z|<1$, and by \hyp{F5} and Corollary \ref{cor:main}-$(c)$, $u_\eps$ is
also uniformly bounded:
 $$\|u_\eps\|_\infty\leq \|f\|_\infty.$$
By \hyp{\Meps} it then follows that 
$$L_{\eps}^{\delta}[u_\eps,D\phi](y_\eps)+\tilde
L_{\eps}^{\delta}[u_\eps,D\phi](y_\eps)=O\left(\int_{|z|\geq\delta}\,d\mu_\eps(z)\right)\to
0\qquad\text{as}\qquad \eps\to0$$ 
and $\delta>0$ is fixed.

Sending first $\eps\to0$ and then $\delta\to0$ in \eqref{v_ineq} then leads
to
$$-L_0\big(D\phi(x),D^2\phi(x)\big)+\bar{u}(x)\leq f(x),$$
and hence $\bar{u}$ is a viscosity subsolution of \eqref{eq:local}. In
a similar way we can show that $\underline{u}$ is a viscosity
supersolution of \eqref{eq:local}. We omit the proof.
\end{proof}

Then it is enough to invoke the comparison principle for the limit (local)
equation to conclude that Theorem~\ref{thm:localization}-$(a)$ holds:

\begin{proof}[Proof of Theorem \ref{thm:localization}-$(a)$]
    By Lemma \ref{lem:limit} $\bar{u}$ is a subsolution and
    $\underline{u}$ is a supersolution of \eqref{eq:local}, hence
    $\bar{u}\leq\underline{u}$ by the comparison principle for
    \eqref{eq:local}, see \cite{cil} or Lemma \ref{comp_loc}
    below. Since $\underline{u}\leq\bar{u}$ by 
    definition,  $\bar{u}=\underline{u}=:u$, and hence $u_\eps\to u$
    point-wise, $u$ is continuous and a viscosity
    solution of \eqref{eq:local}. 
    Now we show that the convergence is locally
      uniform. Fix any $R>0$ and take $x_\eps\in B_R(0)$ such that
    $$\max_{|x|\leq R}\Big(u_\eps(x)-u(x)\Big)=u_\eps(x_\eps)-u(x_\eps)$$
 for any $\eps >0$. Since $|x_\eps|\leq R$, there exists a convergent
 subsequence $x_\eps\to\bx$ for some $|\bx|\leq R$. By the continuity
 of $u$, the definition of 
    $\overline{u}$, and the fact that $u=\overline{u}$, it follows that 
    $$\limsup_{\eps\to0}\max_{|x|\leq R}\Big(u_\eps(x)-u(x)\Big)
    =\limsup_{\eps\to0}\Big(u_\eps(x_\eps)-u(x_\eps)\Big)\leq\overline{u}(\bx)-u(\bx)=0\,.$$ 
A similar argument shows that $\displaystyle\liminf_{\eps\to0}\max_{|x|\leq
  R}(u_\eps(x)-u(x))\geq0$. Combined with similar arguments for
$-(u_\eps-u)$, this shows that
$$\displaystyle\lim_{\eps\to0}\max_{|x|\leq R}|u_\eps(x)-u(x)|=0\qquad\text{for all}\qquad R>0,$$
 and we are done.
\end{proof}

In the fully nonlinear case, the strategy is the same and we
begin with a half-relaxed limit result.
\begin{lemma}\label{lem:limit2} Assume the assumptions of Theorem
    \ref{thm:localization}-$(b)$ hold.
	Then $\bar{u}$ is a viscosity subsolution of \eqref{eq:00} and 
        $\underline{u}$ is a viscosity supersolution of	\eqref{eq:00}.
\end{lemma}

\begin{proof}
    Take a test-function
    $\phi$ such that $\overline{u}-\phi$ has a maximum at $\bx$ that we can
    assume to be strict. Hence $u_\eps-\phi$ also has a maximum at some 
    $x_\eps$, and $x_\eps\to\bx$  and $u_\eps(x_\eps)\to \ol u(\bx)$
    as $\eps\to0$. Hence $\tilde L_{\eps,\delta}[u_\eps,D\phi](x_\eps)\leq
\tilde L_{\eps,\delta}[\phi,D\phi](x_\eps)$, and since $u_\eps$ is a
subsolution  and \hyp{F1} holds, 
 $$F\Big(u_\eps(x_\eps),D\phi(x_\eps),L_{\eps,\delta}[\phi,D\phi](x_\eps)+L_\eps^\delta[u_\eps,D\phi](x_\eps)+\tilde
    L_{\eps,\delta}[\phi,D\phi](x_\eps)+\tilde
    L_\eps^\delta[u_\eps,D\phi](x_\eps)\Big)\leq f(x_\eps)\,.$$
    As in the proof of Lemma \ref{lem:limit}, 
    $$L_\eps\text{-terms}\to
    L_0\big(D\phi(\bx),D^2\phi(\bx)\big)\qquad\text{as}\qquad\eps\to0\,.$$
    By the continuity of $F$, we then send $\eps\to0$ to find that
    $$F\Big(\overline{u}(\bx),D\phi(\bx),
    L_0\big(D\phi(\bx),D^2\phi(\bx)\big)\Big)\leq f(\bx)\,,$$
    which means that $\overline{u}$ is a subsolution of \eqref{eq:00}.
    In a similar way we can show that $\underline{u}$ is a
    supersolution of \eqref{eq:00}.
\end{proof}

Then we prove a comparison result for the limit equation.

\begin{lemma}
    \label{comp_loc} Under the assumptions of Theorem \ref{thm:localization}-$(b)$,
if $u$ is a bounded usc subsolution of \eqref{eq:00} and $v$ 
  is a bounded lsc supersolution of \eqref{eq:00}, then
  $u\leq v$ in $\R^{N}$.  
\end{lemma}

To prove this result, first note that by \hyp{F1} and \hyp{F2}, the
nonlinearity in \eqref{eq:00}, 
\begin{align}
\label{Hdef}
H(x,u,p,X):=F\Big(u,p,L_0(p,X)\Big)-f(x)\,,
\end{align}
is strictly increasing in $u$ and nonincreasing in $X$. Moreover, it is
straightforward to check that we also have the following result. 
\begin{lemma} 
\label{lem:3.14}Assume  \hyp{J4} and \hyp{F3} hold,  $H$ is
  defined in \eqref{Hdef}, and  $M,r,\hat r,R,\eps> 0$. Let
  $x,y\in\R^N$, $|u|\leq M$,  
$p=2(x-y)/\eps^2+o_R(1)$, $|p|<r$,
$X,Y\in\mathcal{S}_N$, and $-\hat r\leq X\leq Y\leq\hat r$ such that 
$$\big(-\frac{8}{\eps^2}+o_R(1)\big)\begin{pmatrix} I & 0\\ 0 & I \end{pmatrix}\leq 
\begin{pmatrix} X & 0 \\ 0 & Y \end{pmatrix}\leq \frac{1}{\eps^2}
\begin{pmatrix} I & -I\\ -I & I \end{pmatrix}+o_R(1)\begin{pmatrix} I
  & 0 \\ 0 & I \end{pmatrix}.$$

Then there are modulii of continuity
$\omega_M,\omega_{M,r,\hat r}$ such that
$$H\Big(y,u,p-O(\tfrac1R),Y\Big)-H\Big(x,u,p+O(\tfrac1R),X\Big)\leq
\omega_M(x-y)
+\frac1{\eps^2}\omega_{M,r,\hat r}\big(o_R(1)\big)+o_R(1)\,.$$ 
\end{lemma}

\begin{proof}[Proof of Lemma \ref{comp_loc}]
In view of our assumptions and Lemma \ref{lem:3.14}, the proof is
standard. It can be obtained by following the line of
reasoning of the proof of Theorem 5.1 in \cite{cil}. We omit the details.
% The idea is to start by considering a maximum point of $\Phi$ defined
% in \eqref{Phi} with a test function $\phi$ as in \eqref{test_func} with
% $\varphi(x)=|x|^2$. In view of the above properties of $H$, the proof
% is standard and essentially follows the proof given in \cite{cil}. The
% difference is the $R$-terms in the test function (which only give rise
% to $o_R(1)$-terms) because we are on an unbounded
% domain, and then the order of the limits: Here we first send
% $R\to\infty$ and then $\eps\to0$ to get the conclusion as in the proof
% of Theorem \ref{thm:comp}.  Thm 5.1
\end{proof}

\begin{proof}[Proof of Lemma \ref{thm:localization}-$(b)$]
In view of Lemma \ref{lem:limit2} and \ref{comp_loc}, we can conclude the
proof exactly as for Theorem~\ref{thm:localization}-$(a)$ above. 
\end{proof}

%%%%%%%%%%%%%%%%%%

%---------------------------------------------
\section{Extensions}
\label{sect:extensions}
%-------------------------------------------

%---------------------------------
\subsection{Parabolic equations}\label{sect:parabolic}
In this section we extend the results to the case of quasilinear and
fully nonlinear parabolic equations,
\begin{align}\label{eq:main.nonlocal.time}
    u_t -L[u,Du]&=f(x,t)\qquad \text{in}\qquad Q_T:=\R^N\times (0,T)\,,\\
\label{eq:0.time}
    u_t+F\big(u,Du,L[u,Du]\big)&=f(x,t)\qquad\text{in}\qquad Q_T\,,
\end{align}
with initial data
\begin{align*}\label{IC}
u(x,0)=u_0(x)\qquad\text{in}\qquad \R^N.
\end{align*}
These extensions are straightforward since the time variable does not play an
important role here. We keep the same notation, definitions, and most
of the assumptions as in the previous sections.  However, we
take the following parabolic versions of \hyp{F2}, \hyp{F4} and \hyp{F5}:

\medskip
\noindent\hyp{F2'}\ assumption  \hyp{F2} holds with $\gamma_M=0$.

\medskip
\noindent\hyp{F6} \ $u_0\in UC(\R^N)$ and $f\in UC(Q_T)$.
\medskip

\noindent\hyp{F6'} \ $u_0$ and $f$ satisfy \hyp{F6}, are bounded, and the
quantities in \hyp{F2}, \hyp{F3} are independent of $M$.

\medskip
As usual for parabolic problems, we do no longer need strict
monotonicity in $u$, see \hyp{F2'}.
We have the following parabolic version of the existence and 
comparison results. 

\begin{theorem} {\rm (Comparison results)}\label{thm:comp.time}

\noindent $(a)$ \brak{Quasilinear case} 
            Assume \hyp{M}, \hyp{J1}--\hyp{J2} and \hyp{F6}. If $u$ is a
            bounded usc subsolution of \eqref{eq:main.nonlocal.time},
            $v$        a bounded lsc supersolution of
            \eqref{eq:main.nonlocal.time}, and $u(x,0)\leq v(x,0)$ in
            $\R^N$, then $u\leq v$ in 
            $\ol Q_T$.\smallskip 

  \noindent $(b)$ \brak{Fully nonlinear case}
    Assume \hyp{M}, \hyp{J1}--\hyp{J2}, \hyp{F1}, \hyp{F2'}, \hyp{F3}, and 
    \hyp{F6}.
If  $u$ is a bounded usc viscosity subsolution of \eqref{eq:0.time},
$v$ a bounded lsc viscosity supersolution of \eqref{eq:0.time}, and
$u(x,0)\leq v(x,0)$ in $\R^N$, then $u\leq  v$ in $\ol Q_T$.
\end{theorem}

\begin{proof}[Sketch of proof] 
%By considering  $e^{\gamma_Mt}u$ instead of $u$, we can assume that
%$\gamma_M=0$ in \hyp{F2'}.
We assume by contradiction that
$$m=\sup_{\R^N\times[0,T]}(u-v)>0.$$
We need to     double the variables in time as well as in space and  consider
    $$\Phi_{\eps,\beta,R,c,\delta}(x,y,s,t):=u(x,t)-v(y,s)-\phi(x,t,y,s)$$
    where
    $$\phi(x,t,y,s)=\frac{1}{\eps^2}\varphi(x-y)+\frac{|t-s|^2}{\beta^2}+
    \psi\Big(\frac{x}{R}\Big)
    + \psi\Big(\frac{y}{R}\Big)+\frac{c}{T-t}+\delta m\frac tT\,,$$
$\delta\in(0,1)$, and $\varphi$ and $\psi$ are defined in the proof of Theorem
\ref{thm:comp}.

A standard argument shows that
$$\sup_{\R^{2N}\times[0,T]^2}\Phi_{\eps,\beta,R,c,\delta}(x,y,t,s)\geq
\sup_{\R^{N}\times[0,T]}\Phi_{\eps,\beta,R,c,\delta}(x,x,t,t)\geq m-o_R(1)-c-\delta m>0$$
 for $c$ small enough and $R$ big enough. By
    definition, $\Phi_{\eps,\beta,R,c}$ will attain its supremum at
    some point $(\bx,\bar t,\bar y,\bar s)$, and
    since $u(x,0)\leq v(x,0)$, we may assume that both
      $\bar t, \bar s >0$ by taking $\eps$ and $\beta$ small enough. Because of the term $c/(T-t)$, we may also
      assume that $\bar t, \bar s<T$. Hence we may use the viscosity
      inequalities for $u$ and $v$. After we have subtracted these
      inequalities and observed that $\delta m\frac1 T \leq 
\phi_t-\phi_s$,
we use the continuity of the equation to send $\beta\to0$ and obtain
$\bar s=\bar t\in(0,T)$. The inequality corresponding to \eqref{diff}
then takes the form
$$\delta m\frac1 T\leq F(v(\by,\bar t),\dots)-F(u(\bx,\bar
t),\dots)+f(\bx,\bar t)-f(\by,\bar t).$$
Since $m>0$, we can use \hyp{F2'} to see that
$$\delta m\frac1 T\leq F(u(\bx,\bar t),\dots)-F(u(\bx,\bar
t),\dots)+f(\bx,\bar t)-f(\by,\bar t)=R.H.S.$$
At this stage we proceed as in the proof of Theorem~\ref{thm:comp} 
(but omitting \hyp{F2}), and show that 
    $$\delta m\frac1 T \leq \lim_{c\to0}\lim_{\eps\to0}\lim_{R\to\infty}\lim_{\beta\to0}
    R.H.S.\leq0\,.$$
This is a contradiction to our original assumption $m>0$ and the proof is complete.
\end{proof}

\begin{theorem}[Existence] \label{thm:existence.time}
    Under the assumptions of Theorem \ref{thm:comp.time}, \hyp{F6'},
    and \hyp{J3}, there 
    exists a bounded viscosity solution of \eqref{eq:0.time}.
\end{theorem}

\begin{proof}[Sketch of proof]
    We first assume \hyp{J1'}, \hyp{J2'}
    and that $(x,t)\mapsto f(x,t)$, $x\mapsto u_0(x)$ are bounded and 
    Lipschitz continuous. The general result will follow from passage to the
    limit as in Proposition~\ref{prop:ex}. 
    Then we  consider the parabolic version of
    the approximation we used in Section~\ref{sect:existence}:
    \begin{equation*}
     \begin{cases}  u_t+T_M\Big[F\big(u,Du,L^R_{k}[u,Du]\big)\Big]
         -\eps\Delta u=f(x,t)\,,&\quad  (x,t)\in B_R(0)\times(0,T)\,,\\
         u(x,t)=0,&  \quad       (x,t)\in \partial B_R(0)\times(0,T)\,,\\
         u(x,0)=u_0(x),& \quad x\in B_R(0)\,.
    \end{cases}
    \end{equation*}
Assuming this problem has a solution $u_{M,R,k,\eps}$, we pass to the limit as
$R\to\infty$, $k\to\infty$, $M\to\infty$,  $\eps\to0$ as in the
elliptic case. Using 
half-relaxed limits and comparison for the limit equation, we show
    that the sequence of solutions has a limit which is a solution of
    \eqref{eq:0.time}. Existence is then proved.

To prove that there exists a solution of the approximate problem we
use Schauder's fixed point theorem and the argument given in
Proposition \ref{lem:existence.approx.pb} with some small modifications: 

    \noindent \textbf{1)} We use Schauder's fixed point theorem 
    in the Banach space $$X:=
          H_{1+\theta_0}\big([0,T]\times\overline{B}_R\big)
%,Du\in  C^{0,\theta_0}\big([0,T]\times\overline{B}_R\big)
$$
    for some $\theta_0\in(0,1)$. The space $H_{1+\theta_0}$
    ($=C^{1+\theta_0,\frac{1+\theta_0}2}$) is a standard  
    parabolic H\"older space where $u\in H_{1+\theta_0}$ e.g. implies
    that $Du\in H_{\theta_0}$. See
    page 46 in \cite{LBook} or Section 1.2.3 in \cite{WYW:Book} for the definition.

    \noindent\textbf{2)} The time-dependent version of Lemma~\ref{app:lem.reg} 
    remains valid: if $v\in X$, then
    $(x,t)\mapsto L^R_k[v,Dv](x,t)$ belongs to
    $H_{\theta_0}(\overline{B}_R\times[0,T])$, and if $v_n\to v$ in 
    $X$, then $L^R_k[v_n,Dv_n]\to L^R_k[v,Dv]$ in
    $H_{\theta_0}(\overline{B}_R\times[0,T])$. 

    \noindent \textbf{3)} The $C^{2,\theta_0}$ regularity result
    that we use in Lemma~\ref{lem:existence.approx.pb} (step 2) is
    replaced by the parabolic $H_{2+\theta_0}$ version in Theorem 4.28 in
    \cite{LBook}. 

    \noindent\textbf{4)} In Lemma~\ref{lem:existence.approx.pb} (step 3),
    instead of the $W^{2,p}$-theory we use the parabolic
    $W^{2,1}_p$-theory of Theorem 7.17 in \cite{LBook}.
We also use the compact embedding of $W^{2,1}_p$ into $X$ for $p$ big
    enough, see Theorem 1.4.1 in \cite{WYW:Book}.
\end{proof}

For the local limit result, we introduce the local parabolic
equations
\begin{align}\label{eq:local.time}
    u_t -L_0(Du,D^2u)&=f(x,t)\qquad \text{in}\qquad Q_T,\\
\label{eq:00.time}
    u_t+F\Big(u,Du,L_0(Du,D^2u)\Big)&=f(x,t)\qquad\text{in}\qquad Q_T\,,
\end{align}
with an initial data $u(x,0)=u_0(x)$. The result is the following:

\begin{theorem} {\rm (Localization)}  \label{thm:localization.time}

\noindent $(a)$ \brak{Quasilinear case} 
            Under the assumptions of Theorem
            \ref{thm:comp.time}-$(a)$, \hyp{F6'}, \hyp{\Meps}, and \hyp{J4}, any
            sequence of solutions  $u_\eps$ of \eqref{eq:lineps} converge locally 
            uniformly in $\ol Q_T$ as $\eps\to0$ to the solution $u$ of
            \eqref{eq:local.time}.\smallskip

  \noindent $(b)$ \brak{Fully nonlinear case}
             Under the assumptions of Theorem \ref{thm:comp.time}-$(b)$, \hyp{F6'}, \hyp{\Meps}, \hyp{J4},  any sequence of solutions  $u_\eps$
      of \eqref{eq:eps} converge locally uniformly in $\ol Q_T$ as
      $\eps\to0$ to the solution $u$ of \eqref{eq:00.time}.
\end{theorem}

\begin{proof}[Sketch of Proof]
    We use uniform boundedness of $u_\eps$ and the half-relaxed limits
\begin{align*}
        \ol{u}(x):=\limsup_{\eps\to0, y\to x,s\to t}u_\eps(y,s)\qquad\text{and}\qquad
        \underline{u}(x):=\liminf_{\eps\to0, y\to x, s\to t}u_\eps(y,s),
\end{align*}
and prove that $\ol{u}$ and $\ul u$ are respectively sub and
supersolutions of the local limit 
problem. Local uniform convergence is then obtained after proving that the limit
problem satisfies the comparison principle. The proofs of the comparison 
principles for \eqref{eq:local.time} and \eqref{eq:00.time} are similar to
the proofs in stationary case with standard modification such as
doubling also the time variables. 
\end{proof}

%---------------------------------------------------
\subsection{More general nonlocal operators and equations}

As is common in viscosity solution theory, our results and proofs
extend easily to equations involving many different 
operators $L$ and equations of Bellman-Isaacs type involving
infima and/or suprema of indexed operators and equations of the
type we have studied before. An example is the following equation:
$$\sup_{\alpha\in\mathcal{A}}\inf_{\beta\in\mathcal{B}}
\bigg\{F^{\alpha,\beta}\Big(u(x),Du(x),
    L_{1,\alpha,\beta}[u,Du](x),\dots, L_{m,\alpha,\beta}[u,Du](x)\Big)-
    f_{\alpha,\beta}(x)\bigg\}=0$$
where 
$$L_{i,\alpha,\beta}[u,Du](x):=
\int_{\R^{P_i}}\Big[\big(u(x+j_{i,\alpha,\beta}(Du,z)\big)-u(x)-
    j_{i,\alpha,\beta}(Du,z)\cdot
    Du(x)\ind{|z|<1}\Big]\d\mu_{i,\alpha,\beta}(z)\,,$$
and for fixed $(i,\alpha,\beta)$, $\mu_{i,\alpha,\beta}$ is a measure on $\R^{P_i}$
    and  $j_{i,\alpha,\beta}:\R^N\times\R^{P_i}\to\R^N$.

\noindent $(i)$ {\em Comparison.} We can extend our comparison 
results easily to this equation if we require $\{j_{i,\alpha,\beta}\}$
and $\{\mu_{i,\alpha,\beta}\}$ to satisfy assumptions \hyp{M} and
\hyp{J1}--\hyp{J3} uniformly with respect to $i$, $\alpha$, and
$\beta$. However, we cannot mix gradient dependence with
$x$-dependence in  $j_{i,\alpha,\beta}$ for reasons explained in the
introduction. 
 This extension is essentially based on the classical inequality 
$$\begin{aligned}
    \sup_{\alpha\in\mathcal{A}}\inf_{\beta\in\mathcal{B}} & \bigg\{
    F^{\alpha,\beta}(u,p,\ell_{\alpha,\beta})-
    f_{\alpha,\beta}(x)\bigg\}-
\sup_{\alpha\in\mathcal{A}}\inf_{\beta\in\mathcal{B}}\bigg\{
    F^{\alpha,\beta}(v,q,\ell'_{\alpha,\beta})-
    f_{\alpha,\beta}(y)\bigg\}\\
& \leq \sup_{(\alpha,\beta)\in\mathcal{A}\times\mathcal{B}} \bigg\{
    F^{\alpha,\beta}(u,p,\ell_{\alpha,\beta})-
    F^{\alpha,\beta}(v,q,\ell'_{\alpha,\beta}) -
    f_{\alpha,\beta}(x)+
    f_{\alpha,\beta}(y)   \bigg\}\,,
\end{aligned}
$$
where for each $(\alpha,\beta)$,
$\ell_{\alpha,\beta},\ell'_{\alpha,\beta}\in\R^m$.
In order to use this inequality in the various passages to the limit,
we have of course to use the uniformity with respect to $i$, $\alpha$, and
$\beta$ of the constants appearing in our hypotheses on $j$ and
$\mu$. Notice that for the $F$-hypotheses, we have to reformulate them with
a vector $\ell\in\R^m$. For instance, ellipticity condition \hyp{F1} becomes

\noindent\hyp{F1'} $F:\R\times\R^N\times\R^m$
is continuous and for any $u\in\R$, $p\in\R^N$, $\ell,\ell'\in\R^m$
s.t. $\ell_i\leq \ell'_i$ ($i=1\dots m$), $$F(u,p,\ell)\leq F(u,p,\ell')\,.$$

\noindent $(ii)$ {\it Existence} is obtained for these more general
equations as we 
did in Section~\ref{sect:existence}, by a series of approximations including
truncations (of the measures and operators), vanishing viscosity and so on. 
Again, uniformity of the hypotheses with respect to $i$,
$\alpha$, and $\beta$ are needed in order to pass to the limit in the various
approximations.

\noindent $(iii)$
{\it The local limit results} follow again the same lines as in
Section~\ref{sect:localization}, though the local operators involve also a sup/inf
of operators $L_0^{i,\alpha,\beta}(Du,D^2u)$ defined in Definition~\ref{defsigma}.

Finally, these extensions are also valid for the parabolic versions discussed in
Section~\ref{sect:parabolic}.

\section*{Acknowlegement}
We would like to thank Luca Rossi for pointing out a critical error in
the existence proof of an earlier version of this paper. 

\appendix

\section{Equi-integrability and convergence}
\label{app:equiint}

We give the definition equi-integrability (also called uniform
integrability) and a result that we have used many times in this paper, a
generalization of the dominated convergence theorem due to Vitali. Our
presentation follow \cite{Ru:Book} page 133.

\begin{definition}
    Let $(\Omega,\mathcal{E},\mu)$ be a positive measure space. A
    family $(f_i)_{i\in 
      I}\subset L^1(\Omega,\mu)$ is equi-integrable if for 
    any $\eps>0$ there exists $\eta>0$ such that
    $$\sup_{i\in I}\int_A|f_i(z)|\,\d\mu(z)<\eps\,$$
 for every $A\in
    \mathcal{E}$ such that $\mu(A)<\eta$. 
\end{definition}

The Vitali convergence theorem is the following result:

\begin{proposition}
Assume $(\Omega,\mathcal{E},\mu)$ is a positive finite measure space and
$(f_n)_{n\in\N}\subset L^1(\Omega,\mu)$ an equi-integrable family such that
$f_n\to f$ $\mu$-a.e.. Then $f_n\to f$ in $L^1(\Omega,\mu)$.
\end{proposition}

Note that the dominated convergence theorem is a consequence of
this result (on finite measure spaces!) since domination by a
fixed, integrable function implies equi-integrability.

\section{Proof of Lemma \ref{lem:measure}}
\label{sec:pf1}

 \noindent $(a)$ \quad
    Note that $\mu_{1}\ind{\frac1k<|z|<k}$ has a finite mass by
    \hyp{M}. An application of Fubini's theorem then shows that the
    convolution $\mu_{1,k}=(\mu_{1}\ind{\frac1k<|z|<k})*\rho_k$ is a measure
    which has density $\bar\mu_{1,k}$ with respect to the Lebesgue
    measure given by 
    $$\bar\mu_{1,k}(x)=\int_{\R^N} \rho_k(x-z)\ind{\frac1k<|z|<k}\,\mu_{1}(\!\dz)\,.$$
    This function is bounded for each $k$:
    $$\|\bar\mu_{1,k}\|_\infty\leq\|\rho_k\|_\infty\, 
    \mu_1\Big(\Big\{\frac1k<|z|<k\Big\}\Big)<\infty\,,$$
A similar argument shows the existence and boundedness of $\bar\mu_{2,k}$.

Let $g(z):=|z|^2\ind{|z|<\delta}$, then by Fubini's theorem and symmetry of $\rho_k$,
    \begin{align*}
        \int_{|z|<\delta}|z|^2\bar\mu_{1,k}(z)\dz &=
    \int_{\R^N}g(z)\int_{\R^N}
    \rho_k(z-y)\ind{\frac1k<|y|<k}\,\mu_{1}(\!\dy)\dz=\int_{\R^N}( g\ast\rho_k) (y)
    \ind{\frac1k<|y|<k}\,\mu_1(\!\dy).
    \end{align*}
Note that $\rho_k*g$ is
    continuous with support in $\{|z|<\delta+1/k\}$. By H\"older's inequality,
$$|\rho_k\ast g|(z)\leq \max_{y\in z+\ol B_{1/k}}|g(y)|\cdot 1 \leq
(|z|+1/k)^2\leq 4|z|^2\qquad\text{for}\qquad |z|>1/k,$$
and hence
\begin{align*}
        \int_{|z|<\delta}|z|^2\bar\mu_{1,k}(z)\dz\leq 4  \int_{|z|<\delta}|z|^2\,\mu_1(\!\dz).
\end{align*}
The proof of $(a)$ is complete.
\medskip

 \noindent $(b)$ \quad Let $\eps>0$ be given, and split the integral in
 two using a $K>\delta$ to separate the domains:
\begin{align*}
I:=&\int_{|z|\geq\delta}\psi_k(z)\, d\mu_{1,k}(z)=\int_{\delta\leq
  |z|\leq K}\psi_k(z)\, d\mu_{1,k}(z) +\int_{|z|\geq K}\psi_k(z)\,
d\mu_{1,k}(z)=I_1+I_2,\\
J:=&\int_{|z|\geq\delta}\psi(z)\, d\mu_{1}(z)=\int_{\delta\leq
  |z|\leq K}\psi(z)\, d\mu_{1}(z) +\int_{|z|\geq K}\psi(z)\,
d\mu_{1}(z)=J_1+J_2.
\end{align*}
We will show that if we take $K$ big enough, then $|I_2|+|J_2|<\eps$
for all $k$, and then if $k$ is big enough, $|I_1-I_2|<\eps$. The
conclusion is that $|I-J|\leq 2\eps$ and the proof is complete.

Consider first $I_2$ and $J_2$. By the definition of $\mu_{1,k}$ and Fubini's
 theorem, $\mu_{1,k}(\{|z|>K\})\leq\mu_1(\{|z|>K-1/k\})$, and then by
 the dominated convergence theorem and $\mu_1(\{|z|>1\})<\infty$,
 $$0\leq\mu_{1,k}(\{|z|>K\})\leq\mu_1(\{|z|>K-1/k\})\to 0\qquad
 \text{as}\qquad K\to\infty.$$
Note that this convergence is uniform in $k$. Hence since $\psi_k$ and
$\psi$ are uniformly bounded in $k$, it follows that $I_2,J_2\to0$ as
$K\to\infty$ uniformly in $k$.

We complete the proof by showing that $|I_1- J_1|\to0$ as
$k\to\infty$ for any fixed $K$. Note that
    $$\begin{aligned}
       |I_1-J_1|\leq  &
    \int_{\delta\leq|z|\leq K}\big|\psi_k-\psi\big|\,\mu_{1,k}(\!\dz) 
    + \Big|\int_{\delta\leq|z|\leq K}\psi\,\mu_{1,k}(\!\dz)-
    \int_{\delta\leq|z|\leq K}\psi\,\mu_{1}(\!\dz)\,\Big|\,.
    \end{aligned}$$
Consider the first term on the right hand side. Since
$\sup_{\delta\leq|z|\leq K}|\psi_k-\psi|\to0$ by assumption, and 
$\mu_{1,k}(\{\delta\leq|z|\leq K\})\leq \mu_1(\{|z|\geq\delta-\frac1k\})$
as in the $|z|>K$ case, for $k>\frac1\delta$ we get 
    $$\int_{\delta\leq|z|\leq K}\big|\psi_k-\psi\big|\,\mu_{1,k}(\!\dz)\leq
    \sup_{\delta\leq|z|\leq
      K}|\psi_k-\psi|\,\mu_1(\{|z|\geq\delta-\tfrac1k\})\to0\quad\text{as}\quad
    k\to\infty\,.$$
For the second term, let $g(z)=\psi(z)\ind{\delta\leq|z|\leq K}$ and
use Fubini's theorem to see that
    $$\int_{\delta\leq|z|\leq K}\psi\,\mu_{1,k}(\!\dz)=
    \int_{\R^N} (\rho_k\ast g)(z)\,\ind{\tfrac1k<|z|<k}\,\mu_{1}(\!\dz)\,.$$
    In the last integral, the support of the convolution is
    $\{\delta-\frac1k\leq|z|\leq K+\frac1k\}$ so we need $k>\frac1\delta$.
    Since $\rho_k\ast g\to g$ and $\ind{\frac1k<|z|<k}\to1$ pointwise
    (almost everywhere) both functions are uniformly bounded, we can
    pass to the limit using dominated convergence and find that
    \begin{equation}\label{app:dom.cv}
    \Big|\int_{\delta\leq|z|\leq K}\psi\,\mu_{1,k}(\!\dz)-
    \int_{\delta\leq|z|\leq K}\psi\,\mu_{1}(\!\dz)\,\Big|\to0\quad\text{as
    }k\to\infty\,.
    \end{equation}
    The proof of $(b)$ is complete.
\medskip

    \noindent $(c)$ \quad
    Let $\eps>0$ be given, and split the integral in
 two using a $0<r<\min(\delta,\delta_0)$ to separate the domains:
\begin{align*}
I:=&\int_{0<|z|\leq\delta}\psi_k(z)\, d\mu_{1,k}(z)=\int_{|z|<r}\psi_k(z)\, d\mu_{1,k}(z) +\int_{r<|z|\leq\delta}\psi_k(z)\,
d\mu_{1,k}(z)=I_1+I_2,\\
J:=&\int_{0<|z|<\delta}\psi(z)\, d\mu_{1}(z)=\int_{|z|<r}\psi(z)\, d\mu_{1}(z) +\int_{r<|z|\leq\delta}\psi(z)\,
d\mu_{1}(z)=J_1+J_2.
\end{align*}
We will show that if we take $r$ small enough, then $|I_2|+|J_2|<\eps$
for all $k$, and then if $k$ is big enough, $|I_1-I_2|<\eps$. The
conclusion is that $|I-J|\leq 2\eps$ and the proof is complete.

The estimate $|I_2|+|J_2|<\eps$ for $r$ small, follows by the
assumptions on $\psi$, part $(a)$, and dominated convergence and \hyp{M}. For
example, 
    $$\int_{0<|z|<r}|\psi_k(z)|\mu_{1,k}(\!\dz)\leq
        C\int_{0<|z|<r}|z|^2\mu_{1,k}(\dz)\leq
        4C\int_{0<|z|<r}|z|^2\mu_{1}(\dz)\to0\quad\text{as}\quad
        r\to 0.$$

Consider now $|I_1-I_2|$. We first introduce the functions
    $$\tilde\psi_k(z):=\frac{\psi_k(z)}{|z|^2}\quad\text{and}\quad
    \tilde\psi(z):=\frac{\psi(z)}{|z|^2}\,,$$
and measures
    $$\tilde\mu_{1,k}(\!\dz):=|z|^2\mu_{1,k}(\!\dz)\quad\text{and}\quad
    \tilde\mu_{1}(\!\dz):=|z|^2\mu_{1}(\!\dz)\,.$$
    By the assumptions and \hyp{M}, $\tilde\psi_k$ and $\tilde\psi$
    are uniformly bounded, $\tilde\psi_k\to\tilde\psi$ uniformly on $r<|z|<\delta$, and
    $\tilde\mu_{1,k}$ and $\tilde\mu_{1}$ are bounded measures on
    $0<|z|<\delta$. It follows that
    \begin{align*}
|I_1-I_2|
%&= \Big|\int_{r<|z|<\delta}\tilde\psi_k\,\tilde\mu_{1,k}(\!\dz)-
%    \int_{r<|z|<\delta}\tilde\psi\,\tilde\mu_{1}(\!\dz)\,\Big|\\
&\leq \int_{r<|z|<\delta}|\tilde\psi_k-\tilde\psi|\,\tilde\mu_{1,k}(\!\dz)+\Big|\int_{r<|z|<\delta}\tilde\psi\,\tilde\mu_{1,k}(\!\dz)-
    \int_{r<|z|<\delta}\tilde\psi\,\tilde\mu_{1}(\!\dz)\,\Big|.
\end{align*}
The first term converges by uniform convergence of $\tilde\psi_k$ and
uniform boundedness of $\mu_{1,k}(\{r<|z|<\delta\})$. For
the second term, we note that (see part $(a)$)
    $$\int_{r<|z|<\delta}\tilde\psi\,\tilde\mu_{1,k}(\!\dz)=
    \int_{r<|z|<\delta}\psi\,\mu_{1,k}(\!\dz)=
    \int_{\R^N} \big(\rho_k\ast(\psi(\cdot)\ind{r<|\cdot|<\delta})\big)(z)\,
   \ind{\frac1k<|z|<k}\ \mu_{1}(\!\dz)\,.$$
    The integrand is uniformly bounded (by
    $\|\psi(\cdot)\ind{r<|\cdot|<\delta}\|_\infty$) and converges pointwise
    to $\psi(z)\ind{r<|z|<\delta}$ for a.a. $z$, so by \hyp{M} and the dominated
    convergence theorem, we can conclude that $|I_1-J_1|\to
    0$ as $k\to\infty$. The proof of $(c)$ is complete. 
\medskip

\noindent $(d)$ \quad This proof is similar to the proof of $(b)$, we omit it.

\section{The proof of Lemma \ref{lem:reg.j}}
\label{sec:pf2}

\noindent  $(a)$\quad Let $K\subset \R^N\times\R^P$ be any bounded set, then by the definition of
$\rho_k$, H\"older's inequality, and $\int \rho_k(p)\,\d p=1$,
$\|j_{i,k}\|_{L^\infty(K)}\leq \|j_{i}\|_{L^\infty(K_{1/k})}$ where
$$K_{1/k}=\big\{(q,z)\in\R^N\times\R^P: \exists p\in\R^N\text{ such
        that }  (p,z)\in K\text{ and }|q-p|<1/k\big\}.$$
 By H\"older's inequality and \hyp{J1'}, we also find that for $|p|\leq
 r$ and $|z|<1$,
 $$|j_{1,k}(p,z)|\leq \max_{|q-p|<1/k}|j(q,z)|\int \rho_k(p)\,\d p
\leq C_{r+1/k}|z|\,,$$
and the proof of $(a)$ is complete.
\smallskip

\noindent $(b)$\quad By similar arguments, for $|p|,|q|,|z|<r$,
$$|j_{i,k}(p,z)-j_{i,k}(q,z)|\leq \int_{|s|<r+1/k}
|j(s,z)||\rho_k(p-s)-\rho_k(q-s)|\,\ds\leq  \sup_{\substack{|s|\leq r+1/k\\
        |z|\leq r}}|j(s,z)|\|D\rho_k\|_{L^1} |p-q|\,,$$
and the proof of $(b)$ is complete by \hyp{J1'} and the standard estimate
$\|D\rho_k\|_{L^1}\leq k \|D\rho\|_{L^1} $.
\smallskip

\noindent $(c)$\quad By the definition of $j_{1,k}$, properties of
mollifiers and Jensen's inequality, Fubini and \hyp{J2},
\begin{align*}
\int_{|z|>0}|j_{1,k}(p,z)-j_{1,k}(q,z)|^2\dmu_1(z)&\leq \int_{|z|>0}
\int_{y\in\R^N}\rho_k(y)\big|j_{1}(p-y,z)-j_{1}(q-y,z)\big|^2dy\,\dmu_1(z)\\ 
&\leq
\int_{y\in\R^N}\rho_k(y)\Big(\int_{|z|>0}\big|j_{1}(p-y,z)-j_{1}(q-y,z)\big|^2
\,\dmu_1(z)\Big)\dy\\
&\leq \int_{y\in\R^N}\rho_k(y)\,\omega_{j,r+\frac1k}(p-q)\,\dy
\leq\omega_{j,r+\frac1k}(p-q)\,.
\end{align*}

\noindent $(d)$\quad 
Let $A\subset \{0<|z|<\delta_0\}$ be a Borel set. Then as in part $(c)$, we use
properties of mollifiers and Jensen's inequality, Fubini and \hyp{J2},
to see that
    \begin{align*}
&\int_A|j_{1,k}(p,z)|^2\,\d\mu(z)\leq
    \int_{A}\int_{|q|<\frac1k}\rho_k(q)\big|j_{1}(p-q,z)\big|^2dq\,\dmu_1(z)
    \leq \max_{|q|<\frac1k}\int_{A}\big|j_{1}(p-q,z)\big|^2\dmu_1(z)\,.
\end{align*}
Now $(d)$ follows from \hyp{J3} applied with $r+1$, which is bigger
than $r+1/k$ for any $k\geq1$.
\smallskip

\noindent $(e)$\quad First fix a $z$ such that $j_i(p,z)$ is
continuous in $p$, \textit{cf.} \hyp{J1}. Hence $j_{i,k}(\cdot,z)$ 
converges locally uniformly to $j_{i}(\cdot,z)$ as $k\to\infty$ (see for instance
Appendix C, Theorem 6 in \cite{E:Book}). Moreover, $j_i(\cdot,z)$ is locally uniformly
continuous in $p$, say with a modulus $\omega_r$ for $|p|,|q|\leq r$. Then
$$|j_{i,k}(p,z)-j_{i,k}(q,z)|\leq \int_{y\in\R^N}
|j(p-y,z)-j(q-y,z)|\rho_{k}(y)dy \leq \omega_r(|p-q|)\cdot 1\,,$$
and $j_{i,k}(\cdot,z)$ is equicontinuous in $p$. Combining these
two results, it follows that for every $p_k\to p$,
\begin{align}\label{j-lim}
j_{i,k}(p_k,z)\to j_i(p,z)\quad \text{as}\quad
k\to\infty\qquad\text{for a.e. $z$}.
\end{align}

Then we let $p_k=D\phi(x_k)$ and $p=D\phi(x)$, and consider 
$L_{2,k}[\phi,p_k](x_k)=\int
\phi(x_k+j_{2,k}(p_k,z))-\phi(x_k)\, \mu_2(\!\dz)$.
By \eqref{j-lim} and continuity of $\phi$ and $D\phi$,
$$\phi(x_k+j_{2,k}(p_k,z))-\phi(x_k)\to
\phi(x+j_{2}(p,z))-\phi(x)\quad\text{for a.e. $z$},$$
and since the integrand is uniformly bounded by the $\mu_2$-integrable function
$2\|\phi\|_\infty$ (\textit{cf.} \hyp{M}), the dominated convergence theorem
implies that $$L_{2,k}[\phi,p_k](x_k)\to L_2[\phi,p](x)\qquad\text{as}\qquad
k\to \infty. $$

Note that $L_{1,k}=L_{1,\delta,k}+L_{1,k}^\delta$. For the
$L_{1,k}^\delta$-term the proof is more or less the same as for the
$L_{2,k}$-term (see above). It only remains to consider the
$L_{1,\delta}$-term.
% \begin{align*}
% &L_{1,\frac1k}[\phi,D\phi(x_k)](x_k)\\
% &=\int_{|z|<\delta}\int_0^1(1-t)j_{1,\frac1k}(D\phi(x_k),z)D^2\phi(x_k+tj_{1,\frac1k}(D\phi(x_k),z))j_{1,\frac1k}(D\phi(x_k),z)\,
% dt\, \mu_1(dz).
% \end{align*}
As above, we see that the integrand converges to
$\phi(x+j_{1}(D\phi(x),z))-\phi(x)- j_{1}(D\phi(x),z)\cdot D\phi(x)$
for a.e. $z$. An application of Taylor's theorem and \hyp{J1'}, show
that the integrand uniformly bounded by the $\mu_1$-integrable
function $\|D^2\phi\|_{L^\infty(B_{R_2})}C_{R_1}|z|^2$ where
$R_1=\max(\delta,\max_k|D\phi(x_k)|)$ and
$$R_2=\max_{k\in\N}|x_k|+\max_{|s|,|z|\leq R_1}|j(s,z)|\,.$$ 
Hence we conclude by the dominated convergence theorem that
$$L_{1,\delta,k}[\phi,D\phi(x_k)](x_k)\to
L_{1,\delta}[\phi,D\phi(x)](x)\qquad\text{as}\qquad k\to \infty.$$
The proof of $(e)$ is complete.

%-----------------------------------------------------------------------------------

\end{document}